\documentclass[
submission
 , nomarks
]{dmtcs-episciences}

\usepackage[utf8]{inputenc}
\usepackage{subfigure}

\usepackage{mathtools}
\usepackage{amsthm} 
\usepackage[all,color]{xy}

\theoremstyle{plain}
\newtheorem{theorem}{Theorem}
\newtheorem{proposition}[theorem]{Proposition}
\newtheorem{lemma}[theorem]{Lemma}
\newtheorem{corollary}[theorem]{Corollary}

\theoremstyle{definition}
\newtheorem{definition}[theorem]{Definition}
\newtheorem{remark}[theorem]{Remark}
\newtheorem{notation}[theorem]{Notation}
\newtheorem{claim}[theorem]{Claim}

\usepackage[round]{natbib}

\author{Narda Cordero-Michel%\affiliationmark{1,2}\thanks{I am fully supported.}
  \and Hortensia Galeana-S\'anchez%\affiliationmark{3}\thanks{And he is, too!}
  }
\title{Pancyclism in the Generalized Sum of Digraphs\thanks{This research was supported by grants CONACYT FORDECYT-PRONACES CF-2019/39570 and UNAM-DGAPA-PAPIIT IN102320.}}
\affiliation{
   Instituto de Matem\'{a}ticas, Universidad Nacional Aut\'{o}noma de M\'{e}xico, M\'{e}xico }
\received{2021-04-06}
\begin{document}
%\publicationdetails{VOL}{2021}{ISS}{NUM}{SUBM}
\maketitle
\begin{abstract}
  A digraph $D=(V,A)$ of order $n\geq 3$ is \emph{pancyclic}, whenever $D$ contains a directed cycle of length $k$ for each $k\in \{3$, \ldots, $n\}$; and $D$ is \emph{vertex-pancyclic} iff, for each vertex $v\in V$ and each  $k\in \{3$, \ldots, $n\}$, $D$ contains a directed cycle of length $k$ passing through $v$. 

Let $D_1$, $D_2$,  \ldots, $D_k$ be a collection of pairwise vertex disjoint digraphs. The \emph{generalized sum} (g.s.) of $D_1$, $D_2$, \ldots, $D_k$, denoted by \begin{math}\oplus_{i=1}^k D_i\end{math} or \begin{math} D_1\oplus D_2 \oplus \cdots \oplus D_k\end{math}, is the set of all digraphs $D$ satisfying:
(i) \begin{math}V(D)=\bigcup_{i=1}^k V(D_i)\end{math},
(ii) \begin{math}D\langle V(D_i) \rangle \cong D_i\end{math} for $i=1$, 2, \ldots, $k$; and
(iii) for each pair of vertices belonging to different summands of $D$, there is exactly one arc between them, with an arbitrary but fixed direction. A digraph $D$ in \begin{math}\oplus_{i=1}^k D_i\end{math} will be called a \emph{generalized sum} (g.s.) of $D_1$, $D_2$,  \ldots, $D_k$.

In this paper we prove that if $D_1$ and $D_2$ are two vertex disjoint Hamiltonian digraphs and $D\in D_1 \oplus D_2$ is strong, then at least one of the following assertions holds: $D$ is vertex-pancyclic, it is pancyclic or it is Hamiltonian and contains a directed cycle of length $l$ for each \begin{math}l\in\{3, \ldots, \max\{|V(D_i)|+1 \colon i\in\{1,2\}\}\}\end{math}. 
Moreover, we prove that if $D_1$, $D_2$,  \ldots, $D_k$ is a collection of pairwise vertex disjoint Hamiltonian digraphs, $n_i=|V(D_i)|$ for each $i\in\{1$, \ldots, $k\}$ and $D \in \oplus_{i=1}^k D_i$ is strong, then at least one of the following assertions holds: $D$ is vertex-pancyclic, it is pancyclic or it is Hamiltonian and contains a directed cycle of length $l$ for each \begin{math}l\in \{3, \ldots, \max\{\left(\sum_{i\in S} n_i\right) + 1 \colon S\subset\{1, \ldots, k\} \text{ with } |S|=k-1\}\}\end{math}.
\end{abstract}

\section{Introduction}
\label{sec:introduction}
Let \begin{math}D=(V(D),A(D))\end{math} be a digraph. Along this paper every directed walk, directed path or directed cycle will simply be called a walk, path or cycle, respectively.
Several authors have studied pancyclic and vertex-pancyclic digraphs and they provided some conditions to determine when a digraph is pancyclic or vertex-pancyclic, as \cite{Bang-Jensen1999313, Bang-Jensen2009, Bang-Jensen1995141, Gutin1995153, Moon1966297, Randerath2002219}; and \cite{Thomassen197785}. In fact, given the difficulty of these matters, some authors have studied partial problems, for instance the $k$-pancyclic digraphs (a digraph $D$ is \emph{$k$-pancyclic} if it contains a cycle of length $l$, for each \begin{math}l\in\{k, k+1, \ldots, |V(D)|\}\end{math}, where \begin{math}3\leq k\leq |V(D)|\end{math}), as \cite{Bang-Jensen1997101, Peters2004227}; and  \cite{Tewes2001193}.

\begin{sloppypar}
Well known results on pancyclism involve large degrees of the vertices or large number of arcs. 
For example,  \cite{Randerath2002219} proved that every digraph $D$ on $n\geq 3$ vertices for which \begin{math}\min\{\delta^+(D), \delta^-(D)\} \geq \frac{n+1}{2}\end{math} is vertex-pancyclic. 
\cite{Haggkvist197620} proved that every Hamiltonian digraph on $n$ vertices and \begin{math}\frac{1}{2}n(n+1)-1\end{math} or more arcs is pancyclic, and  that a strongly connected digraph on $n$ vertices and minimum degree grater than or equal to $n$ is pancyclic unless it is one of the digraphs $K_{p,p}$ ($K_{p,p}$ is a digraph obtained from a complete bipartite graph with $p$ vertices in each partite set, by  replacing each edge with a pair of symmetric arcs); and 
 \cite{Thomassen197785} proved that if $D$ is a strong digraph on $n$ vertices, such that \begin{math}d(x) + d(y) \geq 2n\end{math} is satisfied for each pair of non-adjacent vertices $x$ and $y$, then either $D$ has directed cycles of all lengths 2, 3, \ldots, $n$, or $D$ is a tournament (in which case it has cycles of all lengths 3, 4, \ldots, $n$), or $n$ is even and D is isomorphic to a complete bipartite digraph whose partition sets have $n/2$ vertices.   
Continuing in this direction \cite{Bang-Jensen1999313} proved that any digraph $D$ with no symmetric arcs, $n\geq 9$, minimum degree $n-2$ and such that for each pair of non-adjacent vertices $x$ and $y$ the inequality \begin{math}d^+_D(x)+d^-_D(y)\geq n-3\end{math} holds, is vertex-pancyclic. 
\end{sloppypar}

%Other results on pancyclism have been studied in tournaments and generalizations of tournaments. 
Since it is very difficult to give results on pancyclism for general digraphs, authors have studied the problems of pancyclism and vertex pancyclism in particular families of digraphs, such as tournaments and generalizations of tournaments.
A digraph $D$ is said to be a \emph{tournament} (respectively, a \emph{semicomplete digraph}) whenever for each pair of different vertices, there is exactly one arc (resp. at least one arc) between them. A \emph{$k$-hypertournament} $H$ on $n$ vertices, where \begin{math}2\leq k\leq n\end{math}, is a pair \begin{math}H=(V_H$, $A_H)\end{math}, where $V_H$ is the vertex set of $H$ and $A_H$ is a set of $k$-tuples of vertices such that, for all subsets $S\subseteq V_H$ with $\vert S\vert=k$, $A_H$ contains exactly one permutation of $S$.
A digraph $D$ is a \emph{quasi-transitive digraph} if for every pair of vertices $\{u$, $v\}\subset V(D)$, the existence of a $(u$, $v)$-path of length 2 in $D$ implies that $u$ and $v$ are adjacent. A digraph $D$ is \emph{locally in-semicomplete} (respectively, \emph{locally out-semicomplete}) whenever, for each vertex $v\in V(D)$, the induced subdigraph $D\langle N^-(v)\rangle$ (resp. $D\langle N^+(v)\rangle$) is semicomplete; and $D$ is \emph{locally semicomplete} if it is both locally in- and locally out-semicomplete.
A \emph{locally in-tournament} (respectively, \emph{locally out-tournament}) is a digraph $D$, such that for each vertex $v\in V(D)$, the induced subdigraph $D\langle N^-(v)\rangle$ (resp. $D\langle N^+(v)\rangle$) is a tournament. %; and $D$ is \emph{locally semicomplete} if it is both locally in- and locally out-semicomplete.

\cite{Moon1966297} proved that every strong tournament is vertex-pancyclic; similar results on generalizations of tournaments where obtained by  \cite{Bang-Jensen2009} and by \cite{Li20132749}, where they proved, respectively, that every strong semicomplete digraph is vertex-pancyclic and that every $k$-hypertournament on $n$ vertices, where $3\leq k\leq n-2$, is vertex-pancyclic. \cite{Bang-Jensen1995141} characterized pancyclic and vertex-pancyclic quasi-transitive digraphs and \cite{Bang-Jensen1997101} characterized pancyclic and vertex-pancyclic locally semicomplete digraphs. Other results on locally in-tournament digraphs where obtained by  \cite{Peters2004227} and by  \cite{Tewes2001193,Tewes2002239}. 
Conditions for round decomposable locally semicomplete digraphs and regular multipartite tournaments  to be pancyclic, and vertex-pancyclic where studied by \cite{Bang-Jensen1997101} and by  \cite{Yeo1999137}, respectively, see definitions in the book of \cite{Bang-Jensen2018}.

%In \cite{Yeo1999137}, Yeo proved that every regular multipartite tournament with at least 5 partite sets is vertex-pancyclic and, in \cite{Yeo2007949}, he proved that all regular 4-partite tournaments with at least 13918 vertices are vertex-pancyclic.

Concerning to another generalization of tournaments, \cite{Gutin1995153} studied extended semicomplete digraphs.

Let $R$ be a digraph with vertex set $\{v_1,\ldots, v_n\}$, and let $H_1$, \ldots, $H_n$ be a collection of $n$ pairwise vertex disjoint digraphs. The \emph{composition}, denoted by $R[H_1,\ldots,H_n]$, is the digraph $D$ with vertex set  \begin{math}V(D)=\bigcup_{i=1}^k V(D_i)\end{math} and arc set  \begin{math}A(D)=\left(\bigcup_{i=1}^k A(D_i)\right)\cup \{(u_i,u_j) \colon u_i\in V(H_i), u_j\in V(H_j), (v_i,v_j)\in A(R)\}\end{math}. When $R$ is a tournament (respectively, a semicomplete digraph), $D$ is called an \emph{tournament composition} (resp. a \emph{semicomplete composition}).

If $D = R[H_1,\ldots , H_2]$ and each digraph $H_i$ has empty arc set, then $D$ is an \emph{extension} of $R$. When $R$ is a tournament (respectively, a semicomplete digraph), $D$ is called an \emph{extended tournament} (resp. an \emph{extended semicomplete digraph}).

%A digraph $D$ is called a \emph{complete $k$-partite} or \emph{complete multipartite digraph} (CMD), if its vertex set can be partitioned into $k$ stable  sets (\emph{partite sets}), and for any two vertices $x$ and $y$ in different partite sets either $(x, y)$ or $(y, x)$ (or both) is an arc (are arcs) of $D$. A CMD digraph $D$ is called \emph{ordinary} if for any pair $X$, $Y$ of its partite sets, the set of arcs with both end vertices in $X\cup Y$ coincides with $X\times Y=\{(x,y): x\in X, y\in Y\}$ or $Y\times X$ or $(X \times Y) \cup (Y \times X)$. 

An extended semicomplete digraph $D$ with $k$ partite sets is called a \emph{zigzag digraph} if it has more than four vertices and $k\geq 3$ partite sets $V_1$, $V_2$, $V_3$, \ldots, $V_k$ such that $A(V_2, V_1)=A(V_i, V_2)=A(V_1 , V_i)=\emptyset$ for any $i\in\{3, 4,\ldots, k\}$, \begin{math}|V_1|=|V_2|=|V_3|+|V_4|+\cdots +|V_k|\end{math}. 

Gutin characterized pancyclic and vertex-pancyclic extended semicomplete digraphs:
\begin{theorem}[\cite{Gutin1995153}]

\begin{enumerate}
\item Let $D$ be an extended semicomplete digraph with $k$ partite sets ($k\geq 3$), then $D$  is pancyclic if and only if:
\begin{enumerate}
\item $D$ is strongly connected;
\item it has a spanning subdigraph consisting of a family of vertex disjoint cycles;
\item it is neither a zigzag digraph nor a 4-partite tournament with at least five vertices. 
\end{enumerate}
\item Let $D$ be a pancyclic extended semicomplete digraph with $k$ partite sets, then $D$ is vertex-pancyclic if and only if either:
\begin{enumerate}
\item $k > 3$ or
\item $k=3$ and $D$ has two 2-cycles $Z_1$, $Z_2$ such that $V(Z_1)\cup V(Z_2)$ contains vertices in exactly three partite sets.
\end{enumerate}
\end{enumerate}
\end{theorem}

Let $D$ be a g.s. of $D_1$, $D_2$, \ldots, $D_k$. Observe that, if all exterior arcs between two summands have the same direction, then $D$ is a tournament composition. And if $A(D_i)$ is empty  for each $i\in\{1,2,\ldots,k\}
$ and all exterior arcs between two summands have the same direction, then $D$ is an extended tournament (and thus $D$ is an extended semicomplete digraph). 
%Notice that a g.s. of digraphs with no arcs, where all arcs between any two summands have the same direction, is an extended semicomplete digraph. 
In our results we also work with a vertex partition, but instead of asking that each partite set to be independent, we ask for each partite set to have a Hamiltonian cycle; also, we ask that the arcs between two partite sets be asymmetric and in any direction. In this way, our problem has similarities with Gutin's problem but they are different problems (Figure \ref{fig:not ordinary}).   

The following three theorems are previous results on the existence of cycles in generalized sums of digraphs, they will be very useful in the proof of the main result of the present paper.

%In \cite{Cordero-Michel20161763}, we proved Theorem \ref{theorem strong then Hamiltonian}; and in \cite{Cordero-Michel2020+}, we proved Theorem \ref{theorem strong and no good pair then vertex pancyclic} and a generalization of it for $k$ summands, Theorem \ref{theorem strong and no antidirected 4-cycle then vertex pancyclic}. %, which are stronger results than Theorem \ref{theorem strong then Hamiltonian}.

\begin{theorem}[\cite{Cordero-Michel20161763}]
\label{theorem strong then Hamiltonian}
Let $D_1$, $D_2$, \ldots, $D_k$ be a collection of $k\geq 2$ vertex disjoint Hamiltonian digraphs and $D\in \oplus_{i=1}^k D_i$. If $D$ is strong, then $D$ is Hamiltonian.
\end{theorem}

\begin{figure}
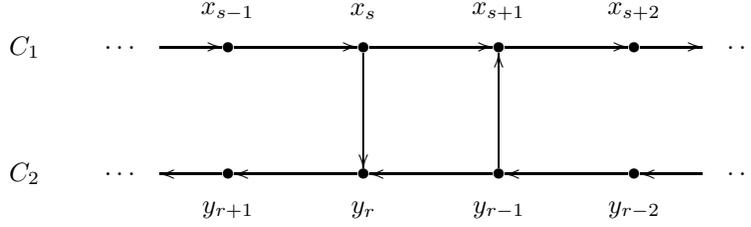

\begin{center}   
  \begin{displaymath}
\xygraph{
	!{<0cm,0cm>;<1.8cm,0cm>:<0cm,1.2cm>::}
 	!{(-0.8,0) }*{\ldots}
	!{(-0.5,0) }*{}="i-1"
	!{(0,0) }*{\bullet}="i0"
 	!{(-0.8,1.4) }*{\ldots}
	!{(-0.5,1.4) }*{}="h-1"
	!{(0,1.4) }*{\bullet}="h0"
	!{(1,0) }*{\bullet}="i1"
	!{(1,1.4) }*{\bullet}="h1"
	!{(2,0) }*{\bullet}="i2"
	!{(2,1.4) }*{\bullet}="h2"
	!{(3,0) }*{\bullet}="i3"
	!{(3.5,0) }*{}="i4"
 	!{(3.8,0) }*{\ldots}
	!{(3,1.4) }*{\bullet}="h3"
 	!{(3.8,1.4) }*{\ldots}
	!{(3.5,1.4) }*{}="h4"
	!{(0,-0.4) }*{y_{r+1}}
	!{(1,-0.4) }*{y_{r}}
	!{(2,-0.4) }*{y_{r-1}}
	!{(3,-0.4) }*{y_{r-2}}
	!{(0,1.8) }*{x_{s-1}}
	!{(1,1.8) }*{x_{s}}
	!{(2,1.8) }*{x_{s+1}}
	!{(3,1.8) }*{x_{s+2}}
	!{(-1.5,1.4) }*{C_1}
	!{(-1.5,0) }*{C_2}
	%!{(4.5,0.7) }*{\sim \text{{\small red}}}
	%!{(4.55,0.3) }*{-\text{ {\small blue}}}
	!~:{@{->}}
	"i0":"i-1"
	"i1":"i0"
	"i2":"i1"
	"i3":"i2"	
	"i4":"i3"
	"h-1":"h0"
	"h0":"h1"
	"h1":"h2"
	"h2":"h3"	
	"h3":"h4"	
	"h1":"i1"	    
    "i2":"h2"
    }
 \end{displaymath}
\caption{A good pair of arcs.}
\label{fig:good pair}
\end{center} 
\end{figure}

\begin{definition}[\cite{Galeana2014315}]
\label{def:good pair}
Let $D$ be a digraph and let $C_1=(x_0$, $x_1$, \ldots, $x_{n-1}$, $x_0)$ and $C_2=(y_0$, $y_1$, \ldots, $y_{m-1}$, $y_0)$ be two vertex disjoint cycles in $D$. A pair of arcs $x_s\to y_r$, $y_{r-1}\to x_{s+1}$ where $s\in \{0$, 1, \ldots, $n-1\}$, $r\in \{0$, 1, \ldots, $m-1\}$, and $s+1$ and $r-1$ are taken modulo $n$ and $m$, respectively,  is  a \emph{good pair of arcs} (Figure \ref{fig:good pair}).

Whenever there is a good pair of arcs between two vertex disjoint cycles $C_1$ and $C_2$, we simply say that there is a good pair.
\end{definition}

\begin{theorem}[\cite{Cordero-Michel2020+}]
\label{theorem strong and no good pair then vertex pancyclic}
Let $D_1$ and $D_2$ be two digraphs with Hamiltonian cycles, $C_1=(x_0$, $x_1$, \ldots, $x_{n-1}$, $x_0)$ and $C_2=(y_0$, $y_1$, \ldots,  $y_{m-1}$, $y_0)$, respectively, and $D\in D_1\oplus D_2$. If $D$ is strong and contains no good pair, then $D$ is vertex-pancyclic.
\end{theorem}

Let $D$ be a digraph. A succession of vertices  $\mathcal{C}=v_0v_1\cdots v_{t-1}v_0$ is an \emph{anti-directed $t$-cycle} whenever $v_i\neq v_j$ for each $i\neq j$, $t$ is even and, for each $i\equiv 0 \pmod{2}$, $\{(v_i, v_{i+1})$, $(v_i, v_{i-1})\}\subset A(D)$ or $\{(v_i, v_{i+1})$, $(v_i, v_{i-1})\}\subset A(D)$. 

We may assume that every anti-directed cycle starts with a forward arc, else we might relabel the subscripts. 
Let $D_1$, $D_2$, \ldots, $D_k$ be a collection of pairwise vertex disjoint digraphs, and $D\in \oplus_{i=1}^k D_i$. 
An anti-directed 4-cycle $\mathcal{C}= v_0 v_1 v_2 v_{3}v_0$ in $D$ will be called a \emph{good cycle} whenever at least one of the following conditions holds \begin{math}\{(v_0,v_1),~(v_2,v_3)\}\subset A(D)\setminus \left(\bigcup_{i=1}^k A(D_i)\right)\end{math} or \begin{math}\{(v_2,v_1),~(v_0,v_3)\}\subset A(D)\setminus \left(\bigcup_{i=1}^k A(D_i)\right)\end{math}.

\begin{theorem}[\cite{Cordero-Michel2020+}]
\label{theorem strong and no antidirected 4-cycle then vertex pancyclic}
Let $D_1$, $D_2$, \ldots, $D_k$ be a collection of $k\geq 2$ vertex disjoint digraphs with Hamiltonian cycles, $C_1$, $C_2$, \ldots, $C_k$, respectively, and $D\in \oplus_{i=1}^k D_i$. If $D$ is strong and contains no good cycle, then $D$ is vertex-pancyclic.
\end{theorem}

In this paper we complete the study of pancyclism and vertex-pancyclism in a strong g.s. of Hamiltonian digraphs. In particular, we consider two vertex disjoint Hamiltonian digraphs, $D_1$ and $D_2$, of order $n_1$ and $n_2$, respectively, where $n_1\geq n_2$ and $d=\gcd(n_1$, $n_2)$. We prove that a strong digraph in $D_1 \oplus D_2$ is a vertex-pancyclic digraph, a pancyclic digraph or a Hamiltonian digraph containing a directed cycle of length $l$ for each \begin{math}l\in\{3, \ldots, n_1+1\}\cup\{n_1+jd \colon 0\leq j < \frac{n_{2}}{d}\}\end{math}. 
We also prove that, if $D_1$, $D_2$,  \ldots, $D_k$ is a collection of pairwise vertex disjoint Hamiltonian digraphs and $n_i=|V(D_i)|$ for each $i\in\{1$, \ldots, $k\}$, then every strong digraph $D\in \oplus_{i=1}^k D_i$ is vertex-pancyclic, pancyclic or Hamiltonian and contains a cycle of length $l$ for each \begin{math}l\in \{3, \ldots, \max\{\left(\sum_{i\in S} n_i\right) + 1 \colon S\subset\{1, \ldots, k\} \text{ with } |S|=k-1\}\}\end{math}.

\begin{figure}
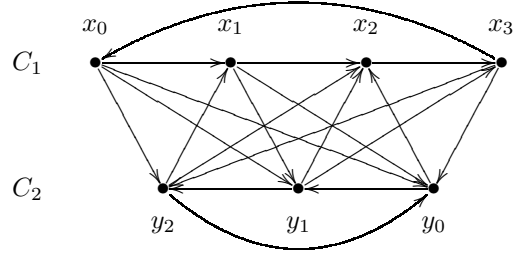

\begin{center}   
 \begin{displaymath}
\xygraph{
	!{<0cm,0cm>;<1.8cm,0cm>:<0cm,1.2cm>::}
	!{(0.5,0) }*{\bullet}="i0"
	!{(0,1.4) }*{\bullet}="h0"
	!{(1.5,0) }*{\bullet}="i1"
	!{(1,1.4) }*{\bullet}="h1"
	!{(2.5,0) }*{\bullet}="i2"
	!{(2,1.4) }*{\bullet}="h2"
	%!{(2.9,0) }*{\bullet}="i3"
	!{(3,1.4) }*{\bullet}="h3"
	%!{(0,-0.4) }*{y_{3}}
	!{(0.5,-0.4) }*{y_{2}}
	!{(1.5,-0.4) }*{y_{1}}
	!{(2.5,-0.4) }*{y_{0}}
	!{(0,1.8) }*{x_{0}}
	!{(1,1.8) }*{x_{1}}
	!{(2,1.8) }*{x_{2}}
	!{(3,1.8) }*{x_{3}}
	!{(-.5,1.4) }*{C_1}
	!{(-.5,0) }*{C_2}
	%!{(4.5,0.7) }*{\sim \text{{\small red}}}
	%!{(4.55,0.3) }*{-\text{ {\small blue}}}
	!~:{@{->}}
	"i1":"i0"
	"i2":"i1"
	%"i3":"i2"	
	"i0"-@/^-0.8cm/@{->}"i2"
	"h0":"h1"
	"h1":"h2"
	"h2":"h3"
	"h3"-@/^-0.8cm/@{->}"h0"	
	"h0":"i0"
	"h0":"i1"
	"h0":"i2"
	%"h0":"i3"
	"i0":"h1"
	"i0":"h2"
	"h3":"i0"
	"i1":"h2"
	"i1":"h3"	
	"h1":"i1"
	"h1":"i2"
	%"i3":"h1"	    
    "i2":"h2"
   %"h2":"i3"
   %"i3":"h3"
    "h3":"i2"
    }
 \end{displaymath}
\caption{This digraph in $C_1\oplus C_2$ is a complete 5-partite strong digraph with partite sets $V_1=\{x_0,x_2\}$, $V_2=\{x_1,x_3\}$, $V_3=\{y_0\}$, $V_4=\{y_1\}$ and $V_5=\{y_2\}$, it has a spanning subdigraph consisting of two vertex disjoint cycles and it is not ordinary (consider the partite sets $V_1$ and $V_2$) nor a zigzag digraph ($|V_1|=|V_2|\neq |V_3|+|V_4|+|V_5|$). In Lemma \ref{lemma one singular vertex} we will see that this digraph is pancyclic.}
\label{fig:not ordinary}
\end{center} 
\end{figure}

\section{Definitions}
\label{sec:definitions}
In this paper $D=(V(D),A(D))$ will denote a \emph{digraph}. An arc $(u,v)\in A(D)$ will also be denoted by $u\to v$. Two different vertices $u$ and $v$ are adjacent if $u\to v$ or $v\to u$. 
Let $A$ and $B$ be two sets of vertices or subdigraphs of a digraph $D$, we define the set of arcs $(A,B)$, as the set of all arcs with tail in $A$ (or in the vertex set of $A$) and head in $B$ (or in the vertex set of $B$). If $A=\{a\}$ or $B=\{b\}$, we use the notation $(a,B)$ or $(A,b)$,  respectively, instead of $(A,B)$. 
Also, we denote by $A\to B$ whenever for each vertex $a$ in $A$ and each vertex $b$ in $B$ we have $a\to b$, and we denote by $A\mapsto B$ whenever $A\to B$ and $(B,A)$ is empty.  If $A=\{a\}$ or $B=\{b\}$, we use the notation $a\to B$ or $A\to b$,  respectively, instead of $A\to B$ and $a\mapsto B$ or $A\mapsto b$,  respectively, instead of $A\mapsto B$. 

The subdigraph induced by a set of vertices $U\subseteq V(D)$ will be denoted by 
$ D\langle U\rangle$;
and if $H$ is a subdigraph of $D$, the subdigraph induced by $V(H)$ will be denoted simply by 
$D\langle H\rangle$. 

A digraph is \emph{strong} whenever for each pair of different vertices $u$ and $v$, there exist a $uv$-path and a $vu$-path.

A \emph{spanning subdigraph} $E$ of $D$ is a subdigraph of $D$ such that $V(E)=V(D)$. We say that $E$ spans $D$.

A \emph{cycle-factor} of a digraph $D$ is a collection $\mathcal{F}$ of pairwise vertex disjoint cycles in $D$ such that each vertex of $D$ belongs to a member of $\mathcal{F}$. A cycle-factor consisting of $k$ cycles is a \emph{$k$-cycle-factor}. 

A path (cycle) in $D$ containing each vertex of $D$ is a \emph{Hamiltonian path} (\emph{Hamiltonian cycle}).

For further details we refer the reader to the book of \cite{Bang-Jensen2009}.

\begin{definition}[\cite{Cordero-Michel20161763}]
\label{definition exterior}
Let $D_1$, $D_2$, \ldots, $D_k$ be a collection of pairwise vertex disjoint digraphs and $D\in \oplus_{i=1}^k D_i$. We will say that $e\in A(D)$ is an \emph{exterior arc} iff \begin{math}e\in A(D)\setminus \left(\bigcup_{i=1}^k A(D_i)\right)\end{math}. %The set of exterior arcs will be denoted by $E_\oplus$.
\end{definition}

\begin{remark}
Clearly the g.s. of two vertex disjoint digraphs if commutative and so is well defined.
Let $D_1$, $D_2$, $D_3$ be three pairwise vertex disjoint digraphs. It is easy to see that the sets \begin{math}(D_1\oplus D_2)\oplus D_3 = \bigcup_{D\in D_1\oplus D_2} D\oplus D_3\end{math} and \begin{math}D_1 \oplus (D_2\oplus D_3)=\bigcup_{D'\in D_2\oplus D_3} D_1\oplus D'\end{math} satisfy \begin{math}\oplus_{i=1}^3 D_i=(D_1\oplus D_2)\oplus D_3 = D_1\oplus(D_2\oplus D_3)\end{math}, and thus the g.s. of three digraphs is well defined and is associative and commutative. By means of an inductive process it is easy to see that the g.s. of $k$ vertex disjoint digraphs is well defined, and is associative and commutative.
\end{remark}

\begin{notation}
Let $k_1$ and $k_2$ be two positive integers, where $k_1\leq k_2$. We will denote by $[k_1,~k_2]$ the set of integers $\{k_1$, $k_1+1$, \ldots, $k_2\}$ when $k_1<k_2$, and $[k_1,~k_2]$ denotes the singleton $\{k_1\}$ when $k_1=k_2$ .
\end{notation}

\begin{remark}[\cite{Cordero-Michel20161763}]
\label{remark subsuma}
Let $D_1$, $D_2$, \ldots, $D_k$ be a collection of pairwise vertex disjoint digraphs, $D\in\oplus_{i=1}^k D_i$ and $J\subset [1,~k]$. The induced subdigraph of $D$ by $\bigcup_{j\in J} V(D_j)$, $H=D\langle \bigcup_{j\in J} V(D_j) \rangle$, belongs to $\oplus_{j\in J} D_j$.
\end{remark}

\begin{notation}
Let $k$ be a positive integer and $A$ be a set of non-negative integers. We will denote by $kA$ the set of integers $\{ka \colon a \in A\}$.
\end{notation}

\vspace{3mm}
%From now on the subscripts for vertices in $C_1=(x_0$, $x_1$, \ldots, $x_{n-1}$, $x_0)$ will be taken modulo $n$ and for vertices in $C_2=(y_0$, $y_1$, \ldots, $y_{m-1}$, $y_0)$ will be taken modulo $m$.
From now on the subscripts for vertices in a cycle $C=(u_0$, $u_1$, \ldots, $u_{t-1}$, $u_0)$ will be taken modulo $l(C)=t$.
\vspace{3mm}

\section{Preliminary results}
\label{sec:preliminary results}

%In this section we will, first, state a result from \cite{Galeana2014315} that we will use along the proofs and then we prove properties of strong digraphs in the g.s. of two Hamiltonian digraphs which allow us to find cycles of several lengths.
Along this section we will use a result by \cite{Galeana2014315}, in order to prove interesting properties of strong digraphs in the g.s. of two Hamiltonian digraphs, which allow us to find cycles of several lengths.

\begin{proposition}[\cite{Galeana2014315}]
\label{propo merging cycles with a good pair of arcs}
Let $C_1$ and $C_2$ be two disjoint cycles in a digraph $D$. If there is a good pair between them, then there is a cycle with vertex set $V(C_1)\cup V(C_2)$ (Figure \ref{fig:good pair of arcs}).
\end{proposition}

\begin{figure}
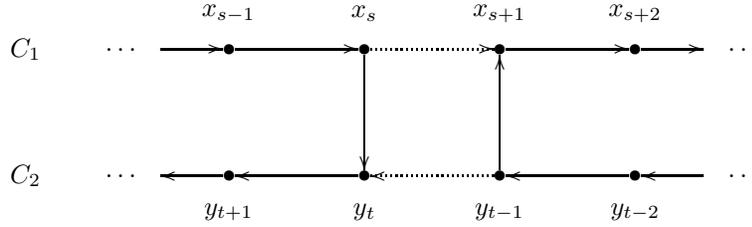

\begin{center}   
 \begin{displaymath}
\xygraph{
	!{<0cm,0cm>;<1.8cm,0cm>:<0cm,1.2cm>::}
 	!{(-0.8,0) }*{\ldots}
	!{(-0.5,0) }*{}="i-1"
	!{(0,0) }*{\bullet}="i0"
 	!{(-0.8,1.4) }*{\ldots}
	!{(-0.5,1.4) }*{}="h-1"
	!{(0,1.4) }*{\bullet}="h0"
	!{(1,0) }*{\bullet}="i1"
	!{(1,1.4) }*{\bullet}="h1"
	!{(2,0) }*{\bullet}="i2"
	!{(2,1.4) }*{\bullet}="h2"
	!{(3,0) }*{\bullet}="i3"
	!{(3.5,0) }*{}="i4"
 	!{(3.8,0) }*{\ldots}
	!{(3,1.4) }*{\bullet}="h3"
 	!{(3.8,1.4) }*{\ldots}
	!{(3.5,1.4) }*{}="h4"
	!{(0,-0.4) }*{y_{t+1}}
	!{(1,-0.4) }*{y_{t}}
	!{(2,-0.4) }*{y_{t-1}}
	!{(3,-0.4) }*{y_{t-2}}
	!{(0,1.8) }*{x_{s-1}}
	!{(1,1.8) }*{x_{s}}
	!{(2,1.8) }*{x_{s+1}}
	!{(3,1.8) }*{x_{s+2}}
	!{(-1.5,1.4) }*{C_1}
	!{(-1.5,0) }*{C_2}
	%!{(4.5,0.7) }*{\sim \text{{\small red}}}
	%!{(4.55,0.3) }*{-\text{ {\small blue}}}
	!~:{@{->}}
	"i0":"i-1"
	"i1":"i0"
	"i3":"i2"	
	"i4":"i3"
	"h-1":"h0"
	"h0":"h1"
	"h2":"h3"	
	"h3":"h4"	
	"h1":"i1"	    
    "i2":"h2"
    !~:{@{.>}}
    "i2":"i1"
    "h1":"h2"
    }
 \end{displaymath}
\caption{A cycle using a good pair of arcs.}
\label{fig:good pair of arcs}
\end{center} 
\end{figure}

\begin{lemma}
\label{lemma subscripts}
Let $D$ be a digraph, $C=(u_0$, $u_1$, \ldots, $u_{t-1}$, $u_0)$ a cycle in $D$, $l$ a positive integer and $d=\gcd(t,~l)$. If $u_s\in V(C)$, then \begin{math}\{u_{s-jl}\}_{j\geq 0}= \{u_{s+il}\}_{i\geq 0}=\{u_{s+id}\}_{i\geq 0}\end{math}.  
\end{lemma}
\begin{proof}
We first prove that $\{u_{s-jl}\}_{j\geq 0}= \{u_{s+il}\}_{i\geq 0}$.
Let $L=\frac{{\rm lcm}(t,~l)}{l}$ and consider $u_{s-jl}$ for some $j\geq 0$. By Euclidean algorithm there exist non-negative integers $p$ and $q$, such that $j=pL+q$, where $0\leq q < L$. Let $i=L-q\geq 0$. Then $s-jl= s-(pL+q)l \equiv s-ql \equiv s+(L-q)l = s+il \pmod{t}$ and thus $u_{s-jl}=u_{s+il}$. %\in \{u_{s+il}\}_{i\geq 0}$. 
Then $\{u_{s-jl}\}_{j\geq 0}\subseteq \{u_{s+il}\}_{i\geq 0}$.

Arguing in a similar way, if we take $u_{s+i'l}$ for some $i'\geq 0$, there are non-negative integers $p'$ and $q'$, such that $i'=p'L+q'$, where $0\leq q' < L$. Let $j'=L-q'\geq 0$. Then $s+i'l= s+(p'L+q')l \equiv s + q'l \equiv s-(L-q')l=s-j'l \pmod{t}$ and thus $u_{s+i'l}=u_{s-j'l}$. %\in \{u_{(s-j(l-2)}$, $ y_{r})\}_{j\geq 0}$.
Then, $\{u_{s-jl}\}_{j\geq 0} \supseteq \{u_{s+il}\}_{i\geq 0}$ and we have the equality.

Now we prove that $\{u_{s+il}\}_{i\geq 0}=\{u_{s+id}\}_{i\geq 0}$. 
Since $d=\gcd(t$, $l)$ we have that $l=hd$ for some $h\geq 1$, and thus we have that $u_{s+il}=u_{s+ihd}$, for each $i\geq 0$. Hence, \begin{math}\{u_{s+il}\}_{i\geq 0}\subset \{u_{s+id}\}_{i\geq 0}\end{math}. We will prove, that both sets have \begin{math}L=\frac{{\rm lcm}(t,~l)}{l}=\frac{tl}{dl}=\frac{t}{d}\end{math} elements.

As $L=\frac{t}{d}$, it follows that $\{u_{s+il}\}_{i\geq 0}=\{u_{s+il}\}_{i=0}^{L-1}$, as $s$, $s+l$, \ldots, $s+(L-1)l$ are different subscripts modulo $t$, 
$s+Ll\equiv s+{\rm lcm}(t,l)\equiv s \pmod{t}$ and all subscripts after $s+Ll$ are also repeated. This is, $\vert\{u_{s+il}\}_{i\geq 0}\vert=L$. Moreover, $\{u_{s+id}\}_{i\geq 0}=\{u_{s+id}\}_{i= 0}^{L-1}$ as $s$, $s+d$, \ldots, $s+(L-1)d$ are different subscripts modulo $t$, $s+Ld=s+t\equiv s \pmod{t}$ and all subscripts after $s+Ld$ are also repeated. Hence,  $\vert\{u_{s+id}\}_{i\geq 0}\vert=L$.
Therefore, $\{u_{s+il}\}_{i\geq 0}= \{u_{s+id}\}_{i\geq 0}$. 
\end{proof}

\begin{remark} 
\label{remark properties 1}
Let $D_1$ and $D_2$ be two vertex disjoint digraphs with Hamiltonian cycles, $C_1=(x_0$, $x_1$, \ldots, $x_{n-1}$, $x_0)$ and $C_2=(y_0$, $y_1$, \ldots, $y_{m-1}$, $y_0)$, respectively, and let $D$ a digraph in $D_1 \oplus D_2$. 
Let $l$ be a fixed integer in $[3$, $ n+1]$ such that $D$ has no cycle of length $l$ and $d=\gcd(n$, $ l-2)$. The following two assertions hold:
\begin{enumerate}[(a)]
\item If $(x_s$, $y_r)\in A(D)$, then $\{(x_{s+k+i(l-2)}$, $y_{r+k})\}_{i\geq 0}\subset A(D)$ for each $k\geq 0$ and $\{(x_{s+k+i(l-2)}$, $y_{r+k})\}_{i\geq 0}=\{(x_{s+k+id}$, $y_{r+k})\}_{i\geq 0}$ for each $k\geq 0$. 

\item If $(y_{r'}$, $x_{s'})\in A(D)$, then $\{(y_{r'+k}$, $x_{s'+k+i(l-2)})\}_{i\geq 0}\subset A(D)$ for each $k\geq 0$ and $\{(y_{r'+k}$, $x_{s'+k+i(l-2)})\}_{i\geq 0}=\{(y_{r'+k}$, $x_{s'+k+id}\}_{i\geq 0}$ for each $k\geq 0$. 
\end{enumerate}
\end{remark}
\begin{proof} 
We will prove the result in three steps.

\begin{claim}
\begin{enumerate}[{Case }1:] % \itemsep-0.1em
\item If $(x_s$, $y_r)\in A(D)$, then $\{(x_{s+i(l-2)}$, $y_{r})\}_{i\geq 0}\subset A(D)$.
\item If $(y_{r'}$, $x_{s'})\in A(D)$, then $\{(y_{r'}$, $x_{s'+i(l-2)})\}_{i\geq 0}\subset A(D)$.
\end{enumerate}
\label{remark1_assertion1}
\end{claim}
\begin{proof}[of claim \ref{remark1_assertion1}]
\begin{enumerate}[(a)]
\item Suppose that $(x_s$, $y_r)\in A(D)$. 
We will prove that $(x_{s-j(l-2)}$, $y_{r})\in A(D)$ for each $j\geq 0$ by induction on $j$; and then we will see that $\{(x_{s-j(l-2)}$, $y_{r})\}_{j\geq 0}=\{(x_{s+i(l-2)}$, $y_{r})\}_{i\geq 0}$.

Since $(x_s$, $y_r)\in A(D)$, the assertion is true for $j=0$.
By the inductive hypothesis we can assume that $(x_{s-j(l-2)}$, $y_r)\in A(D)$. Now we will prove that $(x_{s-(j+1)(l-2)}$, $y_r)\in A(D)$. Indeed, $(x_{s-(j+1)(l-2)}$, $y_r)\in A(D)$, otherwise $(y_r$, $x_{s-(j+1)(l-2)})\in A(D)$ and thus $(x_{s-(j+1)(l-2)}$, $C_1$, $x_{s-j(l-2)})$ $\cup$ $(x_{s-j(l-2)}$, $y_r$, $x_{s-(j+1)(l-2)})$ is a cycle of length $l$ in $D$, which is impossible.
We conclude that $(x_{s-j(l-2)}$, $y_{r})\in A(D)$ for each $j\geq 0$. 

Observe that  $\{x_{s-j(l-2)}\}_{j\geq 0}= \{x_{s+i(l-2)}\}_{i\geq 0}$, by Lemma \ref{lemma subscripts}, and thus $\{(x_{s-j(l-2)}$, $y_{r})\}_{j\geq 0}= \{(x_{s+i(l-2)}$, $y_{r})\}_{i\geq 0}\subset A(D)$.  

%By means of a similar inductive process, it is easy to see that $(y_{r'}$, $x_{s'+i(l-2)})\in A(D)$ for each $i\geq 0$ whenever $(y_{r'}$, $x_{s'})\in A(D)$.

\item Suppose  $(y_{r'},\ x_{s'})\in A(D)$. We will prove that $(y_{r'},\ x_{s'+i(l-2)})\in A(D)$ for each $i\geq 0$ by induction on $i$.

%As $(y_{r'},\ x_{s'})\in A(D)$, we have that $(y_{r'},\ x_{s'+(l-2)})\in A(D)$. Otherwise $(x_{s'+(l-2)},\  y_{r'})\in A(D)$ and thus  $(x_{s'},\ C_1,\ x_{s'+(l-2)})\cup (x_{s'+(l-2)},\ y_{r'},\ x_{s'})$ is a cycle of length $l$ in $D$, which is impossible.
\begin{sloppypar}
Since $(y_{r'}$, $x_{s'})\in A(D)$, the assertion is true for $i=0$.
By the inductive hypothesis we can assume that $(y_{r'}$, $x_{s'+i(l-2)})\in A(D)$. Now we will prove that $(y_{r'}$, $x_{s'+(i+1)(l-2)})\in A(D)$. Indeed, $(y_{r'}$, $x_{s'+(i+1)(l-2)})\in A(D)$, otherwise $(x_{s'+(i+1)(l-2)}$, $y_{r'})\in A(D)$ and thus $(x_{s'+i(l-2)}$, $C_1$, $x_{s'+(i+1)(l-2)})$ $\cup$ $(x_{s'+(i+1)(l-2)}$, $y_{r'}$, $x_{s'+i(l-2)})$ is a cycle of length $l$ in $D$, which is impossible.
We conclude that $(y_{r'}$, $x_{s'+i(l-2)})\in A(D)$ for each $i\geq 0$.
\end{sloppypar}
\end{enumerate}
\end{proof}

\begin{claim} 
\begin{enumerate}[(a)]%\itemsep-0.1em
\item  If $(x_s$, $y_r)\in A(D)$, then $\{(x_{s+k+i(l-2)}$, $y_{r+k})\}_{i\geq 0}\subset A(D)$ for each $k\geq 0$.
\item  If $(y_{r'}$, $x_{s'})\in A(D)$, then $\{(y_{r'+k}$, $x_{s'+k+i(l-2)})\}_{i\geq 0}\subset A(D)$ for each $k\geq 0$.
\end{enumerate}
\label{remark1_assertion2}
\end{claim}
\begin{proof}[of claim \ref{remark1_assertion2}]
\begin{enumerate}[(a)]
\item Suppose $(x_s$, $y_r)\in A(D)$. We prove that $\{(x_{s+k+i(l-2)}$, $y_{r+k})\}_{i\geq 0}\subseteq A(D)$ for each $k\geq 0$, by induction on $k$.

Since $(x_s$, $y_r)\in A(D)$, for $k=0$, it follows from the Assertion \ref{remark1_assertion1} (a) that $\{(x_{s+k+i(l-2)}$, $y_{r+k})\}_{i\geq 0}=\{(x_{s+i(l-2)}$, $y_{r})\}_{i\geq 0}\subseteq A(D)$.

%As $(x_s,\ y_r)\in A(D)$, we have that $(x_{s-(l-3)},\ y_{r+1})\in A(D)$. 
%Otherwise, $(y_{r+1},\ x_{s-(l-3)})\in A(D)$ and thus  $(x_{s-(l-3)},\ C_1,\ x_{s})\cup (x_{s},\ y_r,\ y_{r+1},\ x_{s-(l-3)})$ is a cycle of length $l$ in $D$, which is impossible.
%
%By Remark \ref{remark properties 1}, $(x_{s-(l-3)+i(l-2)},\ y_{r+1})\in A(D)$ for each $i\geq 0$. In particular, when $i=1$, $(x_{s-(l-3)+(l-2)},\ y_{r+1})=(x_{s+1},\ y_{r+1})$. Again, by Remark \ref{remark properties 1}, we have that $\{(x_{s+1+i(l-2)},\ y_{r+1})\}_{i\geq 0}\subset A(D)$.

By the inductive hypothesis we can assume that $\{(x_{s+k+i(l-2)}$, $y_{r+k})\}_{i\geq 0}\subset A(D)$. In particular, for $i=0$, we have that $(x_{s+k}$, $y_{r+k})\in A(D)$. Now we will prove that $\{(x_{s+k+1+i(l-2)}$, $y_{r+k+1})\}_{i\geq 0}\subset A(D)$.

Notice that $(x_{s+k-(l-3)}$, $y_{r+k+1})\in A(D)$, otherwise $(y_{r+k+1}$, $x_{s+k-(l-3)})\in A(D)$ and in this way $(x_{s+k-(l-3)}$, $C_1$, $x_{s+k})\cup (x_{s+k}$, $y_{r+k}$, $y_{r+k+1}$, $x_{s+k-(l-3)})$ is a cycle of length $l$ in $D$, which is impossible.

Since $(x_{s+k-(l-3)}$, $y_{r+k+1})\in A(D)$, we have that $(x_{s+k-(l-3)+i(l-2)}$, $y_{r+k+1})\in A(D)$ for each $i\geq 0$, by Assertion \ref{remark1_assertion1} (a). In particular, for $i=1$, $(x_{s+k-(l-3)+(l-2)}$, $y_{r+k+1})=(x_{s+k+1}$, $y_{r+k+1})\in A(D)$. Again, by Assertion \ref{remark1_assertion1} (a), we obtain that $\{(x_{s+k+1+i(l-2)}$, $y_{r+k+1})\}_{i\geq 0}\subset A(D)$.

Therefore, $\{(x_{s+k+i(l-2)}$, $y_{r+k})\}_{i\geq 0}\subset A(D)$ for each $k\geq 0$.

\item Assume $(y_{r'}$, $x_{s'})\in A(D)$.
To prove that  $\{(y_{r'+k}$, $x_{s'+k+i(l-2)})\}_{i\geq 0} \subseteq A(D)$ for each $k\geq 0$, we first prove that $\{(y_{r'-k'}$, $x_{s'-k'+i(l-2)})\}_{i\geq 0}\subseteq A(D)$ for each $k'\geq 0$ by induction on $k'$.

If $k'=0$, then $\{(y_{r'-k'}$, $x_{s'-k'+i(l-2)})\}_{i\geq 0}=\{(y_{r'}$, $x_{s'+i(l-2)})\}_{i\geq 0}\subseteq A(D)$, by Assertion \ref{remark1_assertion1} (b). 

By the inductive hypothesis we can assume that $\{(y_{r'-k'}$, $x_{s'-k'+i(l-2)})\}_{i\geq 0}\subset A(D)$. Consider $i=\frac{{\rm lcm}(n,\ l-2)}{l-2}-1\geq 0$, then 
$(y_{r'-k'}$, $x_{s'-k'+i(l-2)})=(y_{r'-k'}$, $x_{s'-k'-(l-2)})\in A(D)$, as \begin{math}i(l-2)=\left(\frac{{\rm lcm}(n,~l-2)}{l-2}-1\right)(l-2)={\rm lcm}(n,~l-2)-(l-2)\equiv -(l-2) \pmod{n}\end{math}.

Observe that $(y_{r'-k'-1}$, $x_{s'-k'-(l-2)+(l-3)})=(y_{r'-k'-1}$, $x_{s'-k'-1})\in A(D)$, otherwise $(x_{s'-k'-1}$, $y_{r'-k'-1})\in A(D)$ and thus $(x_{s'-k'-(l-2)}$, $C_1$, $x_{s'-k'-1})$ $\cup$ $(x_{s'-k'-1}$, $y_{r'-k'-1}$, $y_{r'-k'}$, $x_{s-k'-(l-2)})$ is a cycle of length $l$ in $D$, which is impossible. We conclude that $(y_{r'-k'-1}$, $x_{s'-k'-1})=(y_{r'-(k'+1)}$, $x_{s'-(k'+1)})\in A(D)$.
By Assertion \ref{remark1_assertion1} (b), $(y_{r'-(k'+1)}$, $x_{s'-(k'+1)+i(l-2)})\in A(D)$ for each $i\geq 0$. 

Therefore, $\{(y_{r'-k'}$, $x_{s'-k'+i(l-2)})\}_{i\geq 0}\subset A(D)$ for each $k'\geq 0$.

Now we will prove, for each $k\geq 0$, that there exists $k'\geq 0$ such that  $\{(y_{r'+k}$, $x_{s'+k+i(l-2)})\}_{i\geq 0} = \{(y_{r'-k'}$, $x_{s'-k'+j(l-2)})\}_{j\geq 0}$.

Let $L={\rm lcm}(n$, $m)$. By Euclidean algorithm there exist non-negative integers $p$ and $q$ such that $k=pL+q$ where $0\leq q < L$. Define $k'=L-q>0$ and let $F=\{(y_{r'+k}$, $x_{s'+k+i(l-2))})\}_{i\geq 0}$ and $F'= \{(y_{r'-k'}$, $x_{s'-k'+j(l-2))})\}_{j\geq 0}$.

\vspace{2mm}
\begin{claim}
$F=F'$.
\label{claim F=F'}
\end{claim} 
\begin{proof}[of claim  \ref{claim F=F'}]
Let $(y_{r'+k}$, $x_{s'+k+i(l-2)})\in F$ for some $i\geq 0$. Since $k=pL+q$, we have that $(y_{r'+k}$, $x_{s'+k+i(l-2)})=(y_{r'+(pL+q)}$, $x_{s'+(pL+q)+i(l-2)})$.
Observe that $r'+k=r'+pL+q\equiv r'+q \equiv r'-L+q=r'-(L-q)=r'-k' \pmod{m}$ as $m$ divides $L$ and $s'+k+i(l-2)=s'+pL+q+i(l-2)\equiv s'+q+i(l-2)\equiv s'-L+q+i(l-2)=s'-(L-q)+i(l-2)=s'-k'+i(l-2) \pmod{n}$ as $n$ divides $L$. Hence, $(y_{r'+k}$, $x_{s'+k+i(l-2)})=(y_{r'-k'}$, $x_{s'-k'+i(l-2)})\in F'$ and thus $F \subset F'$.

Let $(y_{r'-k'}$, $x_{s'-k'+j(l-2)})\in F'$ for some $j\geq 0$. As $k'=L-q$, we have that $(y_{r'-k'}$, $x_{s'-k'+j(l-2)})=(y_{r'-(L-q)}$, $x_{s'-(L-q)+j(l-2)})$.
Observe that $r'-k'=r'-(L-q)\equiv r'+q \equiv r'+pL+q=r'+k \pmod{m}$ as $m$ divides $L$ and $s'-k'+j(l-2)=s'-(L-q)+j(l-2)\equiv s'+q +j(l-2) \equiv s'+pL+q +j(l-2)=s'+k +j(l-2)\pmod{n}$ as $n$ divides $L$. Hence, $(y_{r'-k'}$, $x_{s'-k'+j(l-2)})=(y_{r'+k}$, $x_{s'+k+j(l-2)})\in F$ and thus $F' \subset F$.
\end{proof}
Therefore, $\{(y_{r'+k}$, $x_{s'+k+i(l-2))})\}_{i\geq 0}\subset A(D)$ for each $k\geq 0$.
\end{enumerate}
%\vspace{2mm}
%Whenever $(x_s$, $y_r)\in A(D)$, it can proved, with a similar inductive process, that $\{(x_{s+k+i(l-2)}$, $y_{r+k})\}_{i\geq 0}\subseteq A(D)$ for each $k\geq 0$.
\end{proof}

\begin{claim} 
\label{remark1_assertion3}
\begin{enumerate}[(a)]% \itemsep-0.1em
\item If $(x_s$, $y_r)\in A(D)$, then $\{(x_{s+k+id}$, $y_{r+k})\}_{i\geq 0}\subset A(D)$ and $\{(x_{s+k+id}$, $y_{r+k})\}_{i\geq 0}= \{(x_{s+k+i(l-2)}$, $y_{r+k})\}_{i\geq 0}$ for each $k\geq 0$. 
\item If $(y_{r'}$, $x_{s'})\in A(D)$, then $\{( y_{r'+k}$, $x_{s'+k+id})\}_{i\geq 0}\subset A(D)$ and $\{(y_{r'+k}$, $x_{s'+k+id})\}_{i\geq 0}= \{( y_{r'+k}$, $x_{s'+k+i(l-2)})\}_{i\geq 0}$ for each $k\geq 0$.
\end{enumerate}

\end{claim}
\begin{proof}[of claim \ref{remark1_assertion3}]
\begin{enumerate}[(a)]
\item Assume that $(x_s$, $y_r)\in A(D)$. Then $\{(x_{s+k+i(l-2)}$, $ y_{r+k})\}_{i\geq 0}\subseteq A(D)$ for each $k\geq 0$, by Assertion \ref{remark1_assertion2} (a).

Take a fixed $k\geq 0$. By Lemma \ref{lemma subscripts}, $\{x_{s+k+id}\}_{i\geq 0}=\{x_{s+k+i(l-2)}\}_{i\geq 0}$. Hence, $\{(x_{s+k+id}$, $y_{r+k})\}_{i\geq 0}=\{(x_{s+k+i(l-2)}$, $y_{r+k})\}_{i\geq 0}$.

\item In a similar way, it can be proved that if $(y_{r'}$, $x_{s'})\in A(D)$, then %$\{( y_{r'+k}$, $x_{s'+k+i(l-2)})\}_{i\geq 0}\subset A(D)$ and 
$\{(y_{r'+k}$, $x_{s'+k+id})\}_{i\geq 0}= \{( y_{r'+k}$, $x_{s'+k+i(l-2)})\}_{i\geq 0} \subset A(D)$ for each $k\geq 0$. 
\end{enumerate}
\end{proof}

From the three claims we have the result.
\end{proof}

Observe that we might exchange the roles of $C_1$ and $C_2$ in Remark \ref{remark properties 1}, asking for $l$ to be a fixed integer in $[3$, $ m+1]$. As a consequence of this, we obtain the following remark: % and we obtain the following:

\begin{remark}
\label{remark properties 4}
Let $D_1$ and $D_2$ be two digraphs with Hamiltonian cycles, $C_1=(x_0$, $x_1$, \ldots, $x_{n-1}$, $x_0)$ and $C_2=(y_0$, $y_1$, \ldots, $y_{m-1}$, $y_0)$, respectively, and let $D$ be a digraph in $D_1 \oplus D_2$. 
Let $l$ be a fixed integer in $[3$, $ m+1]$ such that $D$ has no cycle of length $l$, and $d=\gcd(m$, $ l-2)$. Then the following two assertions hold:
\begin{enumerate}[(a)]
\item \begin{sloppypar} If $(x_s$, $ y_r)\in A(D)$, then $\{(x_{s+k}$, $y_{r+k+i(l-2)})\}_{i\geq 0}\subset A(D)$ for each $k\geq 0$ and $\{(x_{s+k}$, $y_{r+k+i(l-2)})\}_{i\geq 0}= \{(x_{s+k}$, $y_{r+k+id})\}_{i\geq 0}$ for each $k\geq 0$. \end{sloppypar}

\item If $(y_{r'}$, $ x_{s'})\in A(D)$, then $\{(y_{r'+k+i(l-2)}$, $x_{s'+k})\}_{i\geq 0}\subset A(D)$ for each $k\geq 0$ and $\{( y_{r'+k+i(l-2)}$, $x_{s'+k})\}_{i\geq 0} = \{(y_{r'+k+id}$, $x_{s'+k})\}_{i\geq 0}$ for each $k\geq 0$.
\end{enumerate}
\end{remark}

Now we will show the behavior of exterior arcs in a g.s. of two Hamiltonian digraphs when we forbid cycles of length $l$ for some $l\in [n+2$, $ n+m]$.

\begin{remark} 
\label{remark properties 1'}
Let $D_1$ and $D_2$ be two vertex disjoint digraphs with Hamiltonian cycles,  $C_1=(x_0$, $x_1$, \ldots, $x_{n-1}$, $x_0)$ and $C_2=(y_0$, $y_1$, \ldots, $y_{m-1}$, $y_0)$, respectively, and let $D$ be a digraph in $D_1 \oplus D_2$. Let $l$ be a fixed integer in $[n+2$, $ n+m]$ such that $D$ has no cycle of length $l$, $h=l-(n+1)$ and $d=\gcd(n,~m)$. Then the two following assertions hold:
\begin{enumerate}[(a)]
\item If $(x_s$, $y_r)\in A(D)$, then $\{(x_{s+i}$, $y_{r+ih})\}_{i\geq 0}\subset A(D)$ and $\{(x_{s+id}$, $y_{r})\}_{i\geq 0}\subset A(D)$.
\item \sloppy If $(y_{r'}$, $x_{s'})\in A(D)$, then $\{(y_{r'+ih}$, $x_{s'+i})\}_{i\geq 0}\subset A(D)$ and  $\{(y_{r'}$, $x_{s'+id})\}_{i\geq 0}\subset A(D)$.
\end{enumerate}
\end{remark}
\begin{proof}
We will prove the result in two steps.

\begin{claim}
\label{remark2_assertion1}
\begin{enumerate}[(a)] %\itemsep-0.1em
\item If $(x_s$, $y_r)\in A(D)$, then $\{(x_{s+i}$, $y_{r+ih})\}_{i\geq 0}\subset A(D)$.
\item If $(y_{r'}$, $x_{s'})\in A(D)$, then $\{(y_{r'+ih}$, $x_{s'+i})\}_{i\geq 0}\subset A(D)$.
\end{enumerate}
\end{claim}
\begin{proof}[of claim \ref{remark2_assertion1}]
As $l\in [n+2$, $n+m]$, we have that $h=l-(n+1)\in [1$, $m-1]$.

\begin{enumerate}[(a)]
\item Suppose that $(x_s$, $y_r)\in A(D)$. 
We will prove that $(x_{s+i}$, $y_{r+ih})\in A(D)$ for each $i\geq 0$ by induction on $i$.

%As $(x_s,\ y_r)\in A(D)$, we have that $(x_{s+1},\ y_{r+h})$. Otherwise $(y_{r+h},\ x_{s+1})\in A(D)$ and thus  $(x_{s+1},\ C_1,\ x_{s})\cup (x_{s},\ y_r) \cup (y_{r},\ C_2,\ y_{r+h}) \cup (y_{r+h},\ x_{s+1})$ is a cycle of length $(n-1)+1+h+1=l$ in $D$, which is impossible.
As $(x_s$, $y_r)\in A(D)$, the assertion is true for $i=0$.
Assume, by the inductive hypothesis, that $(x_{s+i}$, $y_{r+ih})\in A(D)$ and we will prove that $(x_{s+(i+1)}$, $y_{r+(i+1)h})\in A(D)$. Indeed, $(x_{s+(i+1)}$, $y_{r+(i+1)h})\in A(D)$, otherwise $(y_{r+(i+1)h}$, $x_{s+(i+1)})\in A(D)$ and thus $(x_{s+(i+1)}$, $C_1$, $x_{s+i})$ $\cup$ $(x_{s+i}$, $y_{r+ih})$ $\cup$ $(y_{r+ih}$, $C_2$, $y_{r+(i+1)h})$ $\cup$ $(y_{r+(i+1)h}$, $x_{s+(i+1)})$ is a cycle of length $(n-1)+1+h+1=l$ in $D$, which is impossible.
We conclude that $\mathcal{A}=\{(x_{s+i}$, $y_{r+ih})\}_{i\geq 0}\subset A(D)$. Notice that $(x_{s+(i+1)}$, $C_1$, $x_{s+i})$ is the directed path along the cycle $C_1$ obtained by deleting the arc $(x_{s+i}$, $x_{s+(i+1)})$. 

\item Suppose that $(y_{r'}$, $x_{s'})\in A(D)$.
To prove that $(y_{r'+ih}$, $x_{s'+i})\in A(D)$ for each $i\geq 0$, we will see that $(y_{r'-ih}$, $x_{s'-i})\in A(D)$ for each $i\geq 0$ by induction on $i$; and then we prove that $\{(y_{r'-jh}$, $x_{s'-j})\}_{j\geq 0}=\{(y_{r'+ih}$, $x_{s'+i})\}_{i\geq 0}$.

By our assumption $(y_{r'}$, $x_{s'})\in A(D)$, hence, the assertion is true for $i=0$. By the inductive hypothesis we can assume that $(y_{r'-ih}$, $x_{s'-i})\in A(D)$. 
Now we prove that $(y_{r'-(i+1)h}$, $x_{s'-(i+1)})\in A(D)$. Indeed, $(y_{r'-(i+1)h}$, $x_{s'-(i+1)})\in A(D)$, otherwise $(x_{s'-(i+1)}$, $y_{r'-(i+1)h})\in A(D)$ and thus $(x_{s'-i}$, $C_1$, $x_{s'-(i+1)})$ $\cup$ $(x_{s'-(i+1)}$, $y_{r'+(i+1)h})$ $\cup$ $(y_{r'-(i+1)h}$, $C_2$, $y_{r'-ih})$ $\cup$ $(y_{r'-ih}$, $x_{s'-i})$ is a cycle of length $(n-1)+1+h+1=l$ in $D$, which is impossible.
We conclude that $\{(y_{r'-ih}$, $x_{s'-i})\}_{i\geq 0}\subset A(D)$.
Notice that $(x_{s'-i}$, $C_1$, $x_{s'-(i+1)})$ is the directed path along the cycle $C_1$ obtained by deleting the arc $(x_{s-(i+1)}$, $x_{s-i})$.
\begin{claim} 
\label{remark2_claim}
$\{(y_{r'-jh}$, $x_{s'-j})\}_{j\geq 0}=\{(y_{r'+ih}$, $x_{s'+i})\}_{i\geq 0}$.
\end{claim}
\begin{proof}[of claim \ref{remark2_claim}]
Let $\mathcal{L}={\rm lcm}\left(n,~\frac{{\rm lcm}(m,~h)}{h}\right)$.
Take $(y_{r'-jh}$, $x_{s'-j})$ for some $j\geq 0$. By Euclidean algorithm there exist non-negative integers $p$ and $q$, such that $j=p\mathcal{L}+q$, where $0\leq q < \mathcal{L}$. Let $i=\mathcal{L}-q\geq 0$. Then:
\begin{enumerate}[(i)]
\item Since $n$ divides $\mathcal{L}$, we have that $s'-j = s'-(p\mathcal{L}+q) \equiv s'-q \equiv s'+(\mathcal{L}-q) = s'+i \pmod{n}$, and thus $x_{s'-j}=x_{s'+i}$; 
\item since $m$ divides $\mathcal{L}h$, we have that $r'-jh = r' -(p\mathcal{L}+q)h \equiv r' -qh \equiv r'+(\mathcal{L}-q)h=r'+ih \pmod{m}$, and thus $y_{r'-jh}=y_{r'+ih}$.
\end{enumerate}
Hence, $(y_{r'-jh}$, $x_{s'-j})=(y_{r'+ih}$, $x_{s'+i})\in \{(y_{r'+ih}$, $x_{s'+i})\}_{i\geq 0}$, and  $\{(y_{r'-jh}$, $x_{s'-j})\}_{j\geq 0}\subset \{(y_{r'+ih}$, $x_{s'+i})\}_{i\geq 0}$.

Arguing in a similar way, we can take an arc of the form $(y_{r'+i'h}$, $x_{s'+i'})$ for some $i'\geq 0$. Then, there are non-negative integers $p'$ and $q'$, such that $i'=p'\mathcal{L}+q'$, where $0\leq q' < \mathcal{L}$. Let $j'=\mathcal{L}-q'\geq 0$. Then:
\begin{enumerate}[(i)]
\item Since $n$ divides $\mathcal{L}$, we have that $s'+i'= s+(p'\mathcal{L}+q') \equiv s' + q' \equiv s'-(\mathcal{L}-q')=s-j' \pmod{n}$, and thus $x_{s'+i'}=x_{s'-j'}$;
\item since $m$ divides $\mathcal{L}h$,, we have that $r'+i'h = r' +(p'\mathcal{L}+q')h \equiv r'+q'h \equiv r'-(\mathcal{L}-q')h=r'-j'h \pmod{m}$, and thus $y_{r'+i'h}=y_{r'-j'h}$.
\end{enumerate}
Hence, \begin{math}(y_{r'+i'h},~x_{s'+i'})=(y_{r'-j'h},~x_{s'-j'})\in \{(y_{r'-jh},~x_{s'-j})\}_{j\geq 0}\end{math}, and $\{(y_{r'+ih}$, $x_{s'+i})\}_{i\geq 0}\subset \{(y_{r'-jh}$, $x_{s'-j})\}_{j\geq 0}$.

We conclude that \begin{math}\{(y_{r'-jh},~x_{s'-j})\}_{j\geq 0}=\{(y_{r'+ih},~x_{s'+i})\}_{i\geq 0}\subset A(D)\end{math}.
\end{proof}
\end{enumerate}
\end{proof}

\begin{claim} 
\label{remark2_assertion2}
If $(x_s$, $y_r)\in A(D)$, then $\{(x_{s+id}$, $y_{r})\}_{i\geq 0}\subset A(D)$; and if $(y_{r'}$, $x_{s'})\in A(D)$, then $\{(y_{r'}$, $x_{s'+id})\}_{i\geq 0}\subset A(D)$.
\end{claim}
\begin{proof}[of claim \ref{remark2_assertion2}]
Suppose that $(x_s$, $y_r)\in A(D)$, then $\mathcal{A}=\{(x_{s+i}$, $y_{r+ih})\}_{i\geq 0}\subset A(D)$, by Assertion \ref{remark2_assertion1}. %where $h=l-(n+1)$. 

Consider the following subset of $\mathcal{A}$: $\{(x_{s+(im)}$, $y_{r+(im)h})\}_{i\geq 0}$. As $r+(im)h\equiv r \pmod{m}$ %$y_{r+(im)h}=y_r$ 
we have that $\{(x_{s+(im)}$, $y_{r+(im)h})\}_{i\geq 0}=\{(x_{s+im}$, $y_{r})\}_{i\geq 0}$.

%Let $n'=n/d$, $m'=m/d$ and $L=\lcm[n$, $ m]=\frac{nm}{d}=n'm$, where $d=\gcd(n$, $ m)$. As $m=m'd$, it follows that $\{x_{s+im}\}_{i\geq 0}\subset \{x_{s+id}\}_{i\geq 0}$ and thus $\vert\{x_{s+im}\}_{i\geq 0}\vert \leq \vert\{x_{s+id}\}_{i\geq 0}\vert$. Moreover, $\vert \{x_{s+im}\}_{i\geq 0}\vert=\frac{L}{m}=n'$ and $\vert \{x_{s+id}\}_{i\geq 0} \vert =n'$, as subscripts are taken modulo $n$. Hence, $\{x_{s+im}\}_{i\geq 0}=\{x_{s+id}\}_{i\geq 0}$ and thus $x_{s+id}\to y_{r}$ for each $i\geq 0$.
Since $d=\gcd(n,~m)$, it follows from Lemma \ref{lemma subscripts} that $\{x_{s+im}\}_{i\geq 0}= \{x_{s+id}\}_{i\geq 0}$, and thus $\{(x_{s+im}$, $y_{r})\}_{i\geq 0}=\{(x_{s+id}$, $y_{r})\}_{i\geq 0}\subset A(D)$.

Similarly, whenever $(y_{r'}$, $x_{s'})\in A(D)$, %whenever $\mathcal{B}=\{(y_{r'+ih}$, $ x_{s'+i})\}_{i\geq 0}\subset A(D)$,
it follows that $\{(y_{r'}$, $x_{s'+id})\}_{i\geq 0}\subset A(D)$.
%Now suppose that $(y_{r'}, x_{s'})\in A(D)$, then $\mathcal{B}=\{(y_{r'+ih}$, $ x_{s'+i})\}_{i\geq 0}\subset A(D)$ by Step \ref{remark2_step1}.
%
%Consider the subset $\{(y_{r'+(im)h}$, $ x_{s'+(im)})\}_{i\geq 0}$ of $\mathcal{B}$. As $y_{r'+(im)h}=y_{r'}$ we have that $\{(y_{r'+(im)h}$, $ x_{s'+(im)})\}_{i\geq 0}=\{(y_{r'}$, $ x_{s'+im})\}_{i\geq 0}\subset A(D)$.
%
%Since $\{x_{s+im}\}_{i\geq 0}=\{x_{s+id}\}_{i\geq 0}$, we have that $y_{r'}\to x_{s'+id}$ for each $i\geq 0$.
\end{proof}

\end{proof}

Observe that we might exchange the roles of $C_1$ and $C_2$ in Remark \ref{remark properties 1'}, asking for $l$ to be a fixed integer in $[m+2$, $n+m]$ and $h=l-(m+1)$. Then, we have the following remark:

\begin{remark} 
\label{remark properties 3'}
Let $D_1$ and $D_2$ be two vertex disjoint digraphs with Hamiltonian cycles, $C_1=(x_0$, $x_1$, \ldots, $x_{n-1}$, $x_0)$ and $C_2=(y_0$, $y_1$, \ldots, $y_{m-1}$, $y_0)$, respectively, and let $D$ be a digraph in $D_1 \oplus D_2$. Let $d=\gcd(n,~m)$, $l$ be a fixed integer in $[m+2$, $ n+m]$ such that $D$ has no cycle of length $l$, and $h=l-(m+1)$.
\begin{enumerate}[(a)]
\item If $(x_s, y_r)\in A(D)$, then $\{(x_{s+ih}$, $y_{r+i})\}_{i\geq 0}\subset A(D)$ and $\{(x_{s}$, $y_{r+id})\}_{i\geq 0}\subset A(D)$. 
\item If $(y_{r'}, x_{s'})\in A(D)$, then $\{(y_{r'+i}$, $x_{s'+ih})\}_{i\geq 0}\subset A(D)$ and $\{(y_{r'+id}$, $x_{s'})\}_{i\geq 0}\subset A(D)$.
\end{enumerate}
\end{remark}

Given two Hamiltonian digraphs $D_1$ and $D_2$ of order $n$ and $m$, respectively, and a strong digraph $D$ in $D_1\oplus D_2$. In the present section, we will give sufficient conditions for the existence of cycles of length $l$ for certain $l\in[3$, $ n+m]$.

\begin{definition}
Let $D$ be a digraph, a vertex $v\in V(D)$ is \emph{in-singular} (\emph{out-singular}) with respect to a set of vertices $A\subset V(D)\setminus \{v\}$ if $A\mapsto v$ ($v\mapsto A$); and we will say that $v$ is \emph{singular} with respect to $A$, if it is either in-singular or out-singular with respect to $A$.

If $H$ is a subdigraph of $D$, we will simply say that $v$ is \emph{in-singular}, \emph{out-singular} or \emph{singular} with respect to $H$, whenever $v$ is, respectively, in-singular, out-singular or singular with respect to $V(H)$.
%Let $\mathcal{F'}=\{C, H\}$ be a factor in a digraph $D$, where $C$ is a cycle and $H$ is a subdigraph. A vertex $v\in V(C)$ is \emph{in-singular} (\emph{out-singular}) with respect to $H$ if $H\mapsto v$ ($v\mapsto H$); $v$ is \emph{singular} with respect to $H$, if it is either in-singular or out-singular with respect to $H$.
\end{definition}

In the following lemma we will see that, whenever $D$ is a strong digraph in the g.s. of two Hamiltonian digraphs, $D_1$ and $D_2$, and $D$ contains a singular vertex in $D_i$ with respect to $D_{3-i}$, for some $i\in\{1,2\}$, then $D$ is pancyclic. %contains cycles of each length in $[3$, $ |V(D_1)|+|V(D_2)|]$.

\begin{lemma} 
\label{lemma one singular vertex}
Let $D_1$ and $D_2$ be two Hamiltonian vertex disjoint digraphs   %with Hamiltonian cycles, $C_1=(x_0$, $x_1$, \ldots, $x_{n-1}$, $x_0)$ and $C_2=(y_0$, $y_1$, \ldots, $y_{m-1}$, $y_0)$, respectively,
 and let $D$ be a strongly connected digraph in $D_1 \oplus D_2$. If $D_i$ contains a singular vertex with respect to $D_{3-i}$ in $D$, for some $i\in \{1,2\}$, then $D$ is pancyclic.
\end{lemma}

\begin{proof}
Let $C_1=(x_0$, $x_1$, \ldots, $x_{n-1}$, $x_0)$ and $C_2=(y_0$, $y_1$, \ldots, $y_{m-1}$, $y_0)$ be Hamiltonian cycles in $D_1$ and $D_2$, respectively.

Assume w.l.o.g. that $D_1$ contains a singular vertex, namely $x$, with respect to $D_{2}$. 

\begin{enumerate}[{Case }1:]%[leftmargin=1.5em,itemindent=2.2em] 
\item $x$ is out-singular with respect to $D_2$. 
Then $x \mapsto D_2$, and thus  $x\to y_j$ is an arc in $D$ for each $j\in[0$, $m-1]$ and $(D_2$, $x)=\emptyset$. 

Since $D$ is strong, we have that $(D_2$, $D_1)\neq \emptyset$. Then $(y$, $x')\in (D_2$, $D_1)$ for some $y\in V(D_2)$ and some $x'\in V(D_1)\setminus \{x\}$.

Since $C_1$ is Hamiltonian in $D_1$, $x$ is out-singular with respect to $D_2$ and $(D_2$, $x')\neq \emptyset$, we may find two pairs of consecutive vertices in $C_1$, $x_{s-1}$, $x_{s}$ and $x_q$, $x_{q+1}$ such that $(D_2$, $x_{s-1})\neq \emptyset$, $(D_2$, $x_{q+1})\neq \emptyset$ and  $x_s$ and $x_q$ are both out-singular vertices with respect to $D_2$ (if such two pairs of vertices do not exist, we would contradict the strong connectivity of $D$).

Suppose w.l.o.g. that $(D_2$, $x_{n-1})\neq \emptyset$ and $x_0$ is out-singular with respect to $D_2$ and $x_q$, $x_{q+1}$ satisfy $(D_2$, $x_{q+1})\neq \emptyset$ and $x_q$ is an out-singular vertex with respect to $D_2$.

Let $y_{r}\in V(D_2)$ such that $(y_{r}$, $x_{n-1})\in A(D)$. As $(x_{0}$, $y)\in A(D)$ for each $y\in V(C_2)$ we have that $\alpha_h=(y_{r}$, $x_{n-1}$, $x_{0}$, $y_{r-h})\cup(y_{r-h}$, $C_2$, $y_{r})$ is a cycle in $D$ of length $l(\alpha_h)=3+h$, for each $h\in [0$, $m-1]$. In this way, $D$ contains a cycle of length $l$ for each $l\in[3$, $ m+2]$.

Let $y_{r'}\in V(D_2)$ such that $(y_{r'}$, $x_{q+1})\in A(D)$. As $(x_{q}$, $y)\in A(D)$ for each $y\in V(C_2)$ we have that $\beta_h=(y_{r'}$, $x_{q+1})\cup (x_{q+1}$, $C_1$, $x_{q})\cup (x_q$, $y_{r-h})\cup(y_{r-h}$, $C_2$, $y_{r})$ is a cycle in $D$ of length $l(\beta_h)=1+(n-1)+1+h=n+1+h$, for each $h\in [0$, $m-1]$. In this way, $D$ contains a cycle of length $l$, for each $l\in[n+1$, $ n+m]$.

Observe that, if $n\leq m+2$, we can conclude that $D$ is pancyclic. Then we assume $n>m+2$ and we prove that $D$ contains a cycle of length $l$ for each $l\in [m+3$, $n]$.

Consider a fixed $l\in[m+3$, $n]\subset[m+2$, $m+n]$. Let $h=l-(m+1)$ be an integer in $[2$, $n-m-1]$.

Suppose by contradiction that $D$ has no cycle of length $l$. As $x_0$ is an out singular vertex, we have that $(x_0$, $y_j)\in A(D)$ for each $j\in [0$, $m-1]$ and, by Remark \ref{remark properties 3'} (a), we have that $(x_{ih}$, $y_{j+i})\in A(D)$ for each $i\geq 0$ and each $j\in [0$, $m-1]$. Hence, $\{x_{ih}\}_{i\geq 0}$ is sequence of out-singular vertices which repeats itself after the first time that $ih\equiv 0 \pmod{n}$, this is when $i=\frac{{\rm lcm}(n,\ h)}{h}=\frac{n}{\gcd(n,\ h)}=k $. In this way, $x_0$, $x_h$,\ldots, $x_{(k-1)h}$ are $k$ different out-singular vertices in $D_1$ with respect to $D_2$.

If $\gcd(n$, $h)=1$, then the sequence consists of $n$ different out-singular vertices of $D_1$ with respect to $D_2$ and thus $D_1 \mapsto D_2$ in $D$, a contradiction since $D$ is strong. Therefore, $\gcd(n$, $h)>1$ and $V(D)\setminus \{x_{ih} \colon 0\leq i< k\}$ is non-empty.

Consider the path $P=(x_{n-1}$, $x_0$, \ldots, $x_{h})$, it has length $h+1\in[3$, $n-m]$, $x_{h}$ is an out-singular vertex with respect to $D_2$ and, recall that, $(y_{r}$, $x_{n-1})\in A(D)$.
Then $\gamma= (y_r$, $x_{n-1}) \cup P \cup (x_{h}$, $y_{r+2})\cup (y_{r+2}$, $C_2$, $y_r)$ is a cycle in $D$ of length $l(\gamma)=1+(h+1)+1+(m-2)=m+1+h=l$, a contradiction to our assumption.

Therefore, $D$ must contain a cycle of length $l$ for each $l \in[m+3$, $n]$, which concludes the proof of this case. 

\item $x_0$ is in-singular with respect to $D_2$. 
The proof is similar to that of the previous case.
\end{enumerate}
\end{proof}

In next lemma we will see that, if $D$ is a strong digraph in the g.s. of two Hamiltonian digraphs, $D_1$ and $D_2$, then $D$ contains cycles of each length in $[3$, $|V(D_i)|+1]$ for each $i\in\{1$, $2\}$.

\begin{lemma}
\label{lemma cycle of length <=n+1}
Let $D_1$ and $D_2$ be two Hamiltonian digraphs or order $n$ and $m$, respectively, % with Hamiltonian cycles, $C_1=(x_0$, $x_1$, \ldots, $x_{n-1}$, $x_0)$ and $C_2=(y_0$, $y_1$, \ldots, $y_{m-1}$, $y_0)$, respectively,
 and let $D$ be a strongly connected digraph in $D_1 \oplus D_2$. For each $r\in \{n$, $m\}$ and each integer $l\in[3$, $r+1]$,  $D$ contains a cycle of length $l$.
\end{lemma}
\begin{proof}
As $D$ is strong, if there is a singular vertex in $C_i$ with respect to $C_{3-i}$ for some $i\in\{1$, $2\}$, then Lemma \ref{lemma one singular vertex} asserts that $D$ is pancyclic. So, we assume that $C_i$ has no singular vertex with respect to $C_{3-i}$ for each $i\in\{1$, $2\}$.

%In particular $y_0$ is non-singular with respect to $C_1$ and thus there exist arcs $(y_0,\ x_s)\in A(D)$ and $(x_{s'},\ y_0)\in A(D)$ for some $s,\ s'\in[0,n-1]$. Assume w.l.o.g. that $(y_0,\ x_0)\in A(D)$ and 
In particular $y_0$ is non-singular with respect to $C_1$ and thus there exist different indices $\{i$, $i'\}\subset [0$, $n-1]$ such that $x_i \to y_0$ and $y_{0} \to x_{i'}$. 
 As $C_1$ is a cycle containing $x_i$ and $x_{i'}$, we may find two consecutive vertices in $C_1$, $x_s$ and $x_{s+1}$, such that $(x_s$, $y_0)$ and $(y_0$, $ x_{s+1})$ are both in $A(D)$. Assume w.l.o.g. that $\{(x_{n-1}$, $ y_0)$, $(y_0$, $ x_{0})\} \subset A(D)$

First, consider a fixed length $h\in[3$, $n+1]$. Suppose by contradiction that $D$ has no cycle of length $h$.

\begin{enumerate}[{Case }1:]
\item $\gcd(n$, $ h-2)=1$.
By Remark \ref{remark properties 1}, we have that $(y_{0}$, $ x_{i})\in A(D)$ for each $i\geq 0$, as $\gcd(n$, $ h-2)=1$. Hence, $y_0 \mapsto D_1$ and $y_0$ is a singular vertex with respect to $D_1$, contradicting our assumption. Therefore, $D$ contains a cycle of length $h$.

\item $\gcd(n$, $h-2)>1$. Let $d=\gcd(n$, $h-2)$. By Remark \ref{remark properties 1}, we have that $(x_{n-1+k+id}$, $y_{k})\in A(D)$ and $(y_{k}$, $x_{k+id})\in A(D)$ for each $i\geq 0$ and each $k\geq 0$. %Hence, if $k=hd$ for some $h\geq 0$, we have that $(x_{hd+id-1}$, $ y_{hd})\in A(D)$ and $(y_{hd}$, $ x_{hd+id})\in A(D)$ for each $i\geq 0$ and each $h\geq 0$.
Hence, when $i=0$, we have that $(x_{k-1}$, $y_{k})\in A(D)$ and $(y_{k}$, $x_{k})\in A(D)$ for each $k\geq 0$.

Consider the index $k=n-(h-2)$. It can be written as $n-(h-2)=n'd$ for some $n'\geq 1$ as $d=\gcd(n$, $h-2)$ and $h-2\in [1$, $n-1]$ (and thus $n-(h-2)\in[1$, $n-1]$). Then, the vertex $x_{2+n'd}$ can be written as $x_{2+n'd}=x_{2+n-(h-2)}=x_{n-h+4}$.

In this way, the walk  $\alpha= (x_0$, $y_1$, $x_1$, $y_2$, $x_{2+n'd})$ $\cup$ $(x_{n-h+4}$, $C_1$, $x_{0})$ is a cycle of length $4+(n-(n-h+4))=h$, a contradiction.

Therefore, $D$ contains a cycle of length $h$.

\end{enumerate}

Now, consider a fixed length $h'\in[3$, $m+1]$. Suppose by contradiction that $D$ has no cycle of length $h'$.

\begin{enumerate}[{Case }i:]
\item $\gcd(m$, $h'-2)=1$. As $(y_0$, $ x_0)\in A(D)$, we have by Remark \ref{remark properties 1} that $(y_{i}$, $ x_0)\in A(D)$ for each $i\geq 0$, as $\gcd(m$, $l-2)=1$. Hence, $D_2 \mapsto x_0$ and $x_0$ is a singular vertex with respect to $D_2$, contradicting our assumption. Therefore $D$ contains a cycle of length $h'$.

\item $\gcd(m$, $h'-2)>1$. The existence of a cycle of length $h'$, can be proved in a similar way to Case 2, by considering the vertex $x_0$, which is non-singular with respect to $C_2$, and two vertices $y_r$ and $y_{r'}$ in $C_2$, such that $x_0 \to y_r$ and $y_r'\to x_0$.
\end{enumerate}
\end{proof}

As a consequence of Proposition \ref{propo merging cycles with a good pair of arcs}, Lemma \ref{lemma one singular vertex} and Theorem \ref{theorem strong and no good pair then vertex pancyclic}, we obtain a result by \cite{Cordero-Michel20161763}:

\begin{corollary}
\label{corollary hamiltonian}
Let $D_1$ and $D_2$ be two Hamiltonian digraphs and $D\in D_1\oplus D_2$. If $D$ is strong, then $D$ is Hamiltonian. 
\end{corollary}

\begin{lemma}
\label{lemma cycle of length n+id}
Let $D_1$ and $D_2$ be two digraphs  with Hamiltonian cycles, $C_1=(x_0$, $x_1$, \ldots, $x_{n-1}$, $x_0)$ and $C_2=(y_0$, $y_1$, \ldots, $y_{m-1}$, $y_0)$, respectively, $d=\gcd(n$, $m)$, and let $D$ be a strongly connected digraph in $D_1 \oplus D_2$. For each integer \begin{math}l\in\{n+id \colon 1\leq i \leq \frac{m}{d}\}\end{math}, $D$ contains a cycle of length $l$.
\end{lemma}
\begin{proof}
We may suppose that $C_i$ has no singular vertex with respect to $C_{3-i}$ for each $i\in \{1,2\}$, otherwise Lemma \ref{lemma one singular vertex} asserts that $D$ is pancyclic and thus we have the result.

We can also assume that $D$ contains a good pair of arcs, otherwise Theorem  \ref{theorem strong and no good pair then vertex pancyclic} implies that $D$ is vertex pancyclic. %Hence, $D$ contains a cycle of length $n+m=n+\frac{m}{d}d$.

Suppose w.l.o.g. that $x_{n-1}\to y_1$, $y_0\to x_0$ is a good pair of arcs in $D$ and let $m'=m/d$. 
By Proposition \ref{propo merging cycles with a good pair of arcs}, $D$ contains a cycle $C$ such that $V(C)=V(C_1)\cup V(C_2)$, in this way $C$ has length $l(C)=n+m=n+m'd$.

%In particular, $y_0$ is non-singular with respect to $C_1$ and thus there are indices $s$ and $t$ in $[0,n-1]$ such that $y_0\to x_s$ and $x_r\to y_0$. As $x_s$, $x_r\in V(C_1)$ it is possible to find two consecutive vertices, $x_i$ and $x_{i+1}$, in $C_1$ such that $x_i\to y_0$ and $y_0\to x_{i+1}$. Assume w.l.o.g. that $x_0\to y_0$ and $y_0\to x_{1}$.
 
Now, suppose by contradiction that $D$ contains no cycle of length $l$, for a fixed $l\in\{n+id \colon 0\leq i \leq m'-1\}$.

%As $x_0\to y_0$ we have, by Remark \ref{remark properties 1'}, that $x_0\to y_{id}$ for each $i\geq 0$. 
As $y_0\to x_{0}$ we have, by Remark \ref{remark properties 1'}, that $y_{id}\to x_0$ for each $i\geq 0$. In particular, $y_{id}\to x_0$ for each $i\in[0$, $m'-1]$.
Hence, $\gamma_i=(x_{n-1}$, $y_1)\cup (y_1$, $C_2$, $y_{id})\cup(y_{id}$, $x_0)\cup (x_0$, $C_1$, $x_{n-1})$ is a cycle of length $l(\gamma_i)=1+(id-1)+1+(n-1)=n+id$ in $D$ for each $i\in[0$, $m'-1]$, contradicting our assumption.

% It can be proved in similar way that $D$ contains a cycle of length $l$ for each $l\in\{m+id\colon 0\leq i \leq n'-1\}$.
\end{proof}

\begin{lemma}
\label{lemma mcd(n,m)=1,2}
Let $D_1$ and $D_2$ be two vertex disjoint digraphs  with Hamiltonian cycles, $C_1=(x_0$, $x_1$, \ldots, $x_{n-1}$, $x_0)$ and $C_2=(y_0$, $y_1$, \ldots, $y_{m-1}$, $y_0)$, respectively, $n\geq m$, $d=\gcd(n$, $ m)$, and let $D$ be a strongly connected digraph in $D_1 \oplus D_2$. 
If $d\in \{1$, $2\}$, then $D$ is pancyclic.
\end{lemma}
\begin{proof}
If $D$ has no good pair of arcs, then Theorem  \ref{theorem strong and no good pair then vertex pancyclic} asserts that $D$ is vertex-pacyclic.

If there exists a singular vertex in $D_i$ with respect to $D_{3-i}$, for some $i\in\{1$, $2\}$, then  Lemma \ref{lemma one singular vertex} gives the result. 

Therefore, we may assume that there is no singular vertex in $D_i$ with respect to $D_{3-i}$, for each $i\in\{1$, $2\}$, and that $D$ contains a good pair of arcs. Suppose w.l.o.g. that $x_0\to y_0$ and $y_{m-1}\to x_1$ is a good pair of arcs.

By Lemma \ref{lemma cycle of length <=n+1}, we have that $D$ contains a cycle of length $l$ for each $l\in[3$, $n+1]$.

Suppose by contradiction that $D$ is not pancyclic. Then there exists an integer $l\in [n+2$, $ n+m]$ such that $D$ contains no cycle of length $l$.

As $x_0\to y_0$ we have, by Remark \ref{remark properties 1'}, that $x_{id}\to y_0$ for each $i\geq 0$.

\begin{enumerate}[{Case }1:]
\item $d=1$. Then $x_i\to y_0$, for each $i\geq 0$ and thus $y_0$ is a singular vertex, a contradiction.
\item $d=2$. Then $x_{2i}\to y_0$, for each $i\geq 0$. As $y_0$ is non-singular with respect to $D_1$, $y_0\to x_s$ for some \begin{math} x_s\in V(C_1)\setminus \{x_{2i} \colon 0\leq i \leq \frac{n}{2}-1\}=\{x_{2i+1} \colon 0\leq i\leq  \frac{n}{2}-1\}\end{math}.
Hence, $s=2s'+1$ for some $s'\in[0, \frac{n}{2}-1]$ and $y_0\to x_{2s'+1+2i}$ for each $i\geq 0$, by Remark \ref{remark properties 1'}. Consider $i=n-s'$, as $2s'+1+2(n-s')\equiv 1 \pmod{n}$, it follows that $y_0\to x_1$ and thus $y_{2j}\to x_1$ for each $j\geq 0$, by Remark \ref{remark properties 3'}. 
% $\{x_{2s'+1+2i} \colon  i\geq 0\}=\{x_{2i+1} \colon 0\leq i\leq \frac{n}{2}-1\}$ as indices are taken modulo $n$ and 2 divides $n$. In particular, $y_0\to x_1$ and thus $y_{2j}\to x_1$ for each $j\geq 0$, by Remark \ref{remark properties 3'}. 

Recall that $y_{m-1}\to x_1$,  so it follows that $y_{m-1+2j}\to x_1$ for each $j\geq 0$, by Remark \ref{remark properties 3'}. 
Therefore, $x_1$ is a singular vertex with respect to $C_2$, a contradiction.
\end{enumerate}

\end{proof}

\section{Main results}
\label{sec:main results}
In this section we will see that, given two Hamiltonian digraphs $D_1$ and $D_2$ of order $n$ and $m$, respectively,  and a strong digraph $D$ in $D_1\oplus D_2$, we can determine if $D$ is pancyclic, vertex pancyclic or determine a set of integers $S\subset [3$, $ n+m]$ such that $D$ contains a cycle of length $l$ for each $l\in S$.

\begin{definition}
\label{d-singular}
Let $D_1$ and $D_2$ be two vertex disjoint digraphs with Hamiltonian cycles, $C_1=(x_0$, $x_1$, \ldots, $x_{n-1}$, $x_0)$ and $C_2=(y_0$, $y_1$, \ldots, $y_{m-1}$, $y_0)$, respectively, $n\geq m$, $d=\gcd(n$, $ m)$, and  $D$ a digraph in $D_1 \oplus D_2$. Let $X_i=\{x_j \colon j\equiv i \pmod{d}\}$ and $Y_i=\{y_j \colon j\equiv i \pmod{d}\}$, for each $i\in [0,d-1]$. A vertex $x_s$ (respectively, $y_r$) is \emph{$d$-singular} with respect to $C_2$ (resp. $C_1$) if, for each $i\in [0,d-1]$, either $x_s\mapsto Y_i$ or $Y_i\mapsto x_s$ (resp. either $y_r\mapsto X_i$ or $X_i\mapsto y_r$).
Otherwise,  $x_s$ (resp. $y_r$) is \emph{$d$-non-singular}. 
\end{definition}

\begin{theorem}
\label{theorem at least one d-non-singular}
Let $D_1$ and $D_2$ be two vertex disjoint digraphs with Hamiltonian cycles, $C_1=(x_0$, $x_1$, \ldots, $x_{n-1}$, $x_0)$ and $C_2=(y_0$, $y_1$, \ldots, $y_{m-1}$, $y_0)$, respectively, $n\geq m$, $d=\gcd(n$, $m)$, and let $D$ be a strongly connected digraph in $D_1 \oplus D_2$. If $D$ has at least one $d$-non-singular vertex, then $D$ is pancyclic.
\end{theorem}
\begin{proof}
As $D$ is strong we have, by Lemma \ref{lemma cycle of length <=n+1}, that $D$ contains a cycle of length $l$, for each $l\in[3$, $n+1]$.  

Suppose by contradiction that there is an integer $l\in[n+2$, $n+m]$ such that $D$ has no cycle of length $l$. We will see that all vertices in $C_1$ are $d$-singular with respect to $C_2$ and all vertices in $C_2$ are $d$-singular with respect to $C_1$.

For each $i\in [0,d-1]$, let $X_i=\{x_j \colon j\equiv i \pmod{d}\}$ and $Y_i=\{y_j \colon j\equiv i \pmod{d}\}$. 
Take $x_s\in V(C_1)$ and $r\in[0,d-1]$. 

If $x_s\to y_r$, then $x_s\to y_{r+jd}$ for each $j\geq 0$, by Remark \ref{remark properties 1'}. Consider an index $i\in [0$, $m-1]$ such that $i\equiv r \pmod{d}$, then we have that $i=qd+r$ for some $q\geq 0$ and thus $x_s\to y_i$. Therefore,  $x_s\mapsto Y_r$. 

And if $y_r\to x_s$, then $y_{r+jd}\to x_{s}$ for each $j\geq 0$, by Remark \ref{remark properties 3'}. Consider an index $i\in [0$, $m-1]$ such that $i\equiv r \pmod{d}$ (recall that $r\in[0,d-1]$ and that $d$ divides $m$), then we have that $i=qd+r$ for some $q\geq 0$ and thus $y_i\to x_s$. Therefore,  $Y_r\mapsto x_s$. 

As $r$ is arbitrary, it follows that $x_s$ is a $d$-singular with respect to $C_2$ and, as $x_s$ was taken arbitrarily,  $x_i$ is  a $d$-singular vertex with respect to $C_2$ for each $i\in[0,n-1]$.

It can be proved in a similar way that each vertex in $C_2$ is $d$-singular with respect to $C_1$.

%Take $y_{r'}\in V(C_2)$ and $s'\in[0,d-1]$. 
%
%If $x_{s'}\to y_{r'}$, then $x_{s'+jd}\to y_{r'}$ for each $j\geq 0$, by Remark \ref{remark properties 1'}. Consider an index $i\in [0,n-1]$ such that $i\equiv s' \pmod{d}$, then we have that $i=q'd+s'$ for some $q'\geq 0$ and thus $x_{i}\to y_{r'}$. Therefore,  $X_{s'}\mapsto y_{r'}$. 
%
%And if $y_{r'}\to x_{s'}$, then $y_{r'}\to x_{s+jd}$ for each $j\geq 0$, by Remark \ref{remark properties 3'}. Consider an index $i\in [0,n-1]$ such that $i\equiv s' \pmod{d}$, then we have that $i=q'd+s'$ for some $q'\geq 0$ and thus $y_{r'}\to x_i$. Therefore,  $y_r\mapsto X_{s'}$. 
%
%As $s'$ and $y_{r'}$ were taken arbitrarily, it follows that $y_{r'}$ is a $d$-singular with respect to $C_1$ and thus $y_i$ is  a $d$-singular with respect to $C_1$ for each $i\in[0,m-1]$.

Hence, each vertex in $C_i$ is $d$-singular with respect to $C_{3-i}$ for each $i\in \{1$, $2\}$, contradicting the hypothesis. Then, $D$ contains a cycle of length $l$ for each $l\in[n+2$, $n+m]$ and thus $D$ is pancyclic.
\end{proof}

\begin{definition}
\label{d*-singular}
Let $D_1$ and $D_2$ be two vertex disjoint digraphs with Hamiltonian cycles, $C_1=(x_0$, $x_1$, \ldots, $x_{n-1}$, $x_0)$ and $C_2=(y_0$, $y_1$, \ldots, $y_{m-1}$, $y_0)$, respectively,  $d=\gcd(n$, $ m)$ and let $D$ be a digraph in $D_1 \oplus D_2$. A vertex $x_s$ (respectively, $y_r$) is \emph{$d^*$-singular} with respect to $C_2$ (resp. $C_1$) if there exists an $i\in [0$, $m-1]$ (resp. $i\in[0$, $n-1]$), such that either $x_s\to y_{i+j}$ for each $j\in[0$, $d-1]$ or $y_{i+j}\to x_s$ for each $j\in[0$, $d-1]$ (resp. either $y_r\to x_{i+j}$ for each $j\in[0$, $d-1]$ or $x_{i+j}\to y_r$ for each $j\in[0$, $d-1]$).
%Otherwise,  $x_s$ (resp. $y_r$) is \emph{$d$-non-singular}. 
\end{definition}

\begin{theorem}
\label{thoerem at least one d*-singular}
Let $D_1$ and $D_2$ be two vertex disjoint digraphs with Hamiltonian cycles, $C_1=(x_0$, $x_1$, \ldots, $x_{n-1}$, $x_0)$ and $C_2=(y_0$, $y_1$, \ldots, $y_{m-1}$, $y_0)$, respectively, $n\geq m$, $d=\gcd(n$, $ m)$ and let $D$ be a strongly connected digraph in $D_1 \oplus D_2$. If $D$ has at least one $d^*$-singular vertex in $C_i$ with respect to $C_{3-i}$ for some $i\in\{1$, $2\}$, then $D$ is pancyclic.
\end{theorem}
\begin{proof}
As $D$ is strong, we have by Lemma \ref{lemma cycle of length <=n+1} that $D$ contains a cycle of length $l$ for each $l\in[3$, $n+1]$. 

Suppose by contradiction that there exists an $l\in[n+2$, $ n+m]$ such that $D$ contains no cycle of length $l$.

Assume w.l.o.g. that $C_1$ contains a $d^*$-singular vertex with respect to $C_2$ and w.l.o.g. we suppose that this vertex is $x_0$. Then, there exists an index $i\in [0$, $m-1]$ such that either $x_s\to y_{i+j}$ for each $j\in[0$, $d-1]$ or $y_{i+j}\to x_s$ for each $j\in[0$, $d-1]$. Suppose w.l.o.g. that $i=0$.

\begin{enumerate}[{Case }1:]
\item $x_0\to y_i$ for each $i\in[0$, $d-1]$. As $D$ is strong and contains  no  cycle of length $l$, we have by Remark \ref{remark properties 3'}  that $x_0\to y_{i+jd}$ for each $j\geq 0$ and each $i\in[0$, $d-1]$. Hence, $x_0\mapsto C_2$ and thus $x_0$ is a singular vertex with respect to $C_2$. By Lemma \ref{lemma one singular vertex} we have that $D$ is pancyclic, contradicting our assumption. 

\item $y_i\to x_0$ for each $i\in[0$, $d-1]$. As $D$ is strong and contains  no  cycle of length $l$, we have by Remark \ref{remark properties 3'}  that $y_{i+jd} \to x_0$ for each $j\geq 0$ and each $i\in[0$, $d-1]$. Hence, $C_2\mapsto x_0$ and thus $x_0$ is a singular vertex with respect to $C_2$. By Lemma \ref{lemma one singular vertex} we have that $D$ is pancyclic, contradicting our assumption.
\end{enumerate}

Therefore,  $D$ contains a cycle of length $l$ for each $l\in[M_2+2$, $ n+m]$. Hence, $D$ is pancyclic.
\end{proof}

The following classification theorem is a direct consequence of Lemmas \ref{lemma one singular vertex}, \ref{lemma cycle of length <=n+1}, \ref{lemma cycle of length n+id}, \ref{lemma mcd(n,m)=1,2} and Theorems \ref{theorem strong and no good pair then vertex pancyclic}, \ref{theorem at least one d-non-singular} and \ref{thoerem at least one d*-singular}. 

\begin{theorem}
\label{theorem classification}
Let $D_1$ and $D_2$ be two Hamiltonian digraphs of order $n$ and $m$, respectively; $n\geq m$; $d=\gcd(n$, $m)$; and $D$ a strongly connected digraph in $D_1\oplus D_2$. Then one of the following assertions holds:
\begin{enumerate}[(i)] %\itemsep-0.1em
\item $D$ is vertex-pancyclic;

\item $D$ is pancyclic; or

\item $D$ is Hamiltonian and it contains a cycle of length $l$ for each \begin{math}l\in [3,~n+1]\cup \{n+id \colon 0\leq i < m/d\}\end{math}.
\end{enumerate}
\end{theorem}

The following theorem is an extension of Theorem \ref{theorem classification} for a g.s. of $k$ Hamiltonian digraphs, which is the strongest result of this paper.

\begin{theorem}
\label{theorem g.s. of n Hamiltonians}
Let $D_1$, $D_2$,  \ldots, $D_k$ be a collection of pairwise vertex disjoint Hamiltonian digraphs, $n_i=|V(D_i)|$ for each $i\in[1$, $k]$, and $D$ a strongly connected digraph in $\oplus_{i=1}^k D_i$. Then one of the following assertions holds: %Then $D$ is pancyclic or $D$ is Hamiltonian and it contains a cycle of length $l$ for each $l\in [3$, $\max\{\left(\sum_{i\in S} n_i\right) + 1 \colon S\subset[1$, $k]$ with $|S|=k-1\}]$.
\begin{enumerate}[(i)]
\item $D$ is vertex-pancyclic;

\item $D$ is pancyclic; or

\item $D$ is Hamiltonian and it contains a cycle of length $l$ for each \begin{math}l\in [3,~ \max\{\left(\sum_{i\in S} n_i\right) + 1 \colon S\subset[1,~k] \text{ with } |S|=k-1\}]\end{math}.
\end{enumerate}
\end{theorem}
\begin{proof}
We will proceed by induction on $k$.

If $k=2$, then Theorem \ref{theorem classification} asserts that $D$ is vertex-pancyclic, $D$ is pancyclic or $D$ is Hamiltonian and it contains a cycle of length $l$ for each $l\in[3$, $N+1]$, where $N=\max\{n_1$, $n_2\}$.

Suppose that the hypothesis holds for each $k'$, with $2\leq k'\leq k-1$. This is: if $D_1$, $D_2$,  \ldots, $D_{k'}$ are $k'$ pairwise vertex disjoint Hamiltonian digraphs, $n_i=|V(D_i)|$ for each $i\in[1$, $k']$, and $D'$ is a strong digraph in $\oplus_{i=1}^{k'} D_i$. Then $D'$ is vertex-pancyclic, $D'$ is pancyclic or $D'$ is Hamiltonian and it contains a cycle of length $l$ for each \begin{math}l\in [3,~\left(\sum_{i\in S'} n_i\right) + 1]\end{math}, for each $S'\subset [1$, $k']$ such that $|S'|=k'-1$.
 
Let $D$ be as in the hypothesis and suppose w.l.o.g. that $n_1\leq n_j$ for each $j\in [2$, $k]$. By Theorem \ref{theorem strong then Hamiltonian}, we know that $D$ is Hamiltonian.

Let $H$ be the digraph with vertex set $V(H)=\{D_1$, $D_2$, \ldots, $D_{k}\}$ and such that $D_i\to D_j$ iff $(D_i$, $D_j)\neq \emptyset$ in $D$. In this way, $H$ is a strong semicomplete digraph. % (and thus it is vertex-pancyclic \cite{Bang-Jensen2009}). 
Take $H'=H\langle \{D_2$, \ldots, $D_{k}\} \rangle$. 

\begin{enumerate}[{Case }1:]
\item $H'$ contains a cycle $\alpha=(D_{i_0}$, $D_{i_1}$, \ldots, $D_{i_{r-1}}$, $D_{i_0})$ of length $r\in[2$, $k-1]$ (we will consider that a pair of symmetric arcs in $H'$ forms a cycle of length 2). By Remark \ref{remark subsuma}, the subdigraph of $D$ induced by $\bigcup_{j=0}^{r-1} V(D_{i_j})$, namely \begin{math}D_0=D\langle \bigcup_{j=0}^{r-1} V(D_{i_j})\rangle \end{math}, is strong and $D_0\in \oplus_{j=0}^{r-1} D_{i_j}$. By Theorem \ref{theorem strong then Hamiltonian}, $D_0$ is Hamiltonian. Now, let \begin{math}J=[0,~k]\setminus \{i_0, \ldots, i_{r-1}\}\end{math}, $n_0=|V(D_0)|$ and notice that $D\in \oplus_{j\in J} D_j$, as $D$ satisfies the definition of g.s. of the $D_j$'s with $j\in J$. Moreover, each summand is a Hamiltonian digraph and $D$ is strong. Hence, by induction hypothesis, $D$ is vertex-pancyclic, $D$ is pancyclic or $D$ is Hamiltonian and it contains a cycle of length $l$ for each $l\in [3$, $\sum_{j\in S} n_j + 1]$, for each $S\subset J$ with $|S|=|J|-1$ elements. In particular, as $D_1\notin V(\alpha)$, this is true when $S= J\setminus \{1\}$. 
Observe that \begin{math}\sum_{j\in J\setminus\{1\}} n_j = n_0 + \sum_{j\in [2,~k]\setminus \{i_0, \ldots, i_{r-1}\}} n_j = \sum_{j\in \{i_0, \ldots, i_{r-1}\}} n_j + \sum_{j\in [2,~k]\setminus \{i_0, \ldots, i_{r-1}\}} n_j = \sum_{j=2}^k D_j\end{math}.
Therefore, $D$ contains a cycle of length $l$ for each \begin{math}l\in [3,~\sum_{j=2}^k n_j + 1]\end{math}, which gives the result as $n_1\leq n_j$ for each $j\in [2$, $k]$.

\item  $H'$ is acyclic (and contains no symmetric arcs). Then $H'$ is an acyclic tournament and thus $H'$ is transitive and contains a unique Hamiltonian path $P$. Suppose w.l.o.g. that $P=(D_2$, \ldots, $D_{k})$. Then, for each pair of different indices $\{i$, $j\}\subset[2$, $k]$, we have that $D_i \mapsto D_j$ in $D$ iff $2\leq i < j\leq k$ (see Chapter 2 of the book of \cite{Bang-Jensen2018}). This is, for each pair of different indices $\{i$, $j\}\subset[2$, $k]$, each $u\in V(D_i)$ and each $v\in V(D_j)$, we have that $(u$, $v)\in A(D)$ iff $2\leq i < j\leq k$.

For each $i\in [1$, $k]$, let $C_i=(x_0^i$, $x_1^i$, \ldots, $x_{n_i-1}^i$, $x_0^i)$ be a Hamiltonian cycle in $D_i$ and $P_i=(x_0^i$, $x_1^i$, \ldots, $x_{n_i-1}^i)$ the Hamiltonian path obtained from $C_i$ by removing the last arc $(x_{n_i-1}^i$, $x_0^i)$.  
Observe that for each pair of different indices $\{i$, $j\}\subset[2$, $k]$, each $x_r^i\in V(D_i)$ and each $x_s^j\in V(D_j)$, we have that $(x_r^i$, $x_s^j)\in A(D)$ iff $2\leq i < j\leq k$.
%Let $e_i=(x_{n_i-1}^i$, $x_{0}^{i+1})$ be the arc from the last vertex in $P_i$ to the first vertex in $P_{i+1}$, for each $i\in[2$, $k]$ (where $x_{0}^{k+1}=x_{0}^{1}$).

Since $D$ is strong, necessarily $(D_1$, $D_2)\neq \emptyset$ and $(D_k$, $D_1)\neq \emptyset$ in $D$. Suppose w.l.o.g. that $(x_{r}^1$, $x_{0}^2)\in A(D)$, for some $r\in [0$, $n_1-1]$, and $(x_{n_k-1}^k$, $x_{0}^1)\in A(D)$. 

\begin{enumerate}[{Case 2.}1:]
\item  $(D_2$, $D_1)= \emptyset$. Hence,  $D_1 \mapsto D_2$ and thus $(x_t^1$, $x_s^2)\in A(D)$, for each $t\in [0$, $n_1-1]$ and each $s\in[0$, $n_2-1]$.  % Let $e_1=(x_{n_1-1}^1$, $x_{0}^{2})$.
We will construct cycles of each length in $[3$, $\sum_{i=1}^k n_i]$, by taking  the Hamiltonian paths $P_1$, \ldots, $P_{j-1}$,  a subpath of length $i$ of  $P_j$,  concatenate them  with the exterior arcs mentioned above and close the cycle with an arc from a vertex in $P_j$ to  $x_{n_k-1}^k$ followed by the arc  $(x_{n_k-1}^k$, $x_{0}^1)$.  Consider the following cycles:
\begin{itemize}
\item $\beta(1,i)=(x_0^1$, $P_1$, $x_i^1)\cup (x_i^1$, $x_0^2$, $x_{n_k-1}^k$, $x_{0}^1)$ is a cycle in $D$ of length $l(\beta(1,i))=i+3$ for each $i\in[0$, $n_{1}-1]$; 
%%%%%\item $\gamma(i,2)=P_1 \cup e_1\cup (x_0^2$, $C_2$, $x_i^2)\cup (x_i^2$, $x_{n_k-1}^k$, $x_{0}^1)$ is a cycle of length $l(\gamma(i,2))=n_1+i+2$ for each $i\in[1$, $n_{2}]$;
%\item $\gamma(j,i)=P_1 \cup e_1 \cup P_2 \cup \cdots \cup P_{j-1}\cup e_{j-1}\cup (x_0^j$, $P_j$, $x_i^j)\cup (x_i^j$, $x_{n_k-1}^k$, $x_{0}^1)$ is a cycle in $D$ of length $l(\gamma(j,i))=\sum_{l=1}^{j-1} n_l+i+2$ for each $j\in [2$, $k-1]$ and each $i\in[1$, $n_{j}-1]$;
\item $\gamma(j,i)=P_1 \cup (x_{n_1-1}^1$, $x_{0}^{2}) \cup P_2 \cup \cdots \cup P_{j-1}\cup (x_{n_{j-1}-1}^{j}$, $x_{0}^{j}) \cup (x_0^j$, $P_j$, $x_i^j)\cup (x_i^j$, $x_{n_k-1}^k$, $x_{0}^1)$ is a cycle in $D$ of length \begin{math}l(\gamma(j,i))=\sum_{l=1}^{j-1} n_l+i+2\end{math} for each $j\in [2$, $k-1]$ and each $i\in[1$, $n_{j}-1]$;
%\item $\eta(k,i)=P_1 \cup e_1 \cup P_2 \cup \cdots \cup P_{k-1}\cup (x_{n_{k-1}-1}^{k-1}$, $x_{n_k-(i+1)}^{k})\cup (x_{n_k-(i+1)}^k$, $C_k$, $x_{n_k-1}^k)\cup e_k$ is a cycle in $D$ of length $l(\eta(k,i))=\sum_{l=1}^{k-1} n_l+i+1$ for each $i\in[1$, $n_k-1]$;
\item $\eta(k,i)=P_1 \cup (x_{n_1-1}^1$, $x_{0}^{2}) \cup P_2 \cup \cdots \cup P_{k-1}\cup (x_{n_{k-1}-1}^{k-1}$, $x_{n_k-(i+1)}^{k})\cup (x_{n_k-(i+1)}^k$, $P_k$, $x_{n_k-1}^k)\cup (x_{n_k-1}^k$, $x_{0}^{1})$ is a cycle in $D$ of length \begin{math}l(\eta(k,i))=\sum_{l=1}^{k-1} n_l+i+1\end{math} for each $i\in[1$, $n_k-1]$;
\end{itemize}
Therefore, $D$ is pancyclic.
\item  $(D_1$, $D_k)= \emptyset$. Then $D_k \mapsto D_1$ and, by similar constructions to those of the previous case, it is possible to prove that $D$ is pancyclic. %It is possible to prove that $D$ is pancyclic, considering construction cycles similar to those of the previous case.  

\item  $(D_2$, $D_1)\neq \emptyset$ and $(D_1$, $D_k)\neq \emptyset$. First, we will construct cycles of each length in $[r+3$, $\sum_{i=2}^k n_i + r + 1]$, by taking  the Hamiltonian paths $P_2$, \ldots, $P_{j-1}$,  a subpath of length $i$ of  $P_j$,  concatenate them  with the exterior arcs mentioned above and, to close the cycle, go to vertex $x_{n_k-1}^k$ by means of an exterior arc and then pass through the path $(x_{n_k-1}^k$, $x_{0}^1$, $x_1^1$, \ldots,  $x_{r}^1$, $x_{0}^2)$ (recall that  $(x_{r}^1$, $x_{0}^2)\in A(D)$, for some $r\in [0$, $n_1-1]$). Consider the following cycles:
\begin{itemize}
\item $\varphi(2,i)= (x_0^2$, $P_2$, $x_i^2)\cup (x_i^2$, $x_{n_k-1}^k$, $x_{0}^1) \cup (x_{0}^1$, $P_1$, $x_{r}^1) \cup (x_{r}^1$, $x_{0}^2) $ is a cycle of length $l(\varphi(2,i))=i+2+r+1$ for each  $i\in[0$, $n_{2}-1]$;
%\item $\varphi(j,i)= P_2 \cup \cdots \cup P_{j-1}\cup e_{j-1}\cup (x_0^j$, $P_j$, $x_i^j)\cup (x_i^j$, $x_{n_k-1}^k$, $x_{0}^1) \cup (x_{0}^1$, $P_1$, $x_{r}^1) \cup (x_{r}^1$, $x_{0}^2) $ is a cycle of length $l(\varphi(j,i))=\sum_{l=2}^{j-1} n_l+i+2+r+1$ for each $j\in [3$, $k-1]$ and each $i\in[0$, $n_{j}-1]$;
\item $\varphi(j,i)= P_2 \cup \cdots \cup P_{j-1}\cup (x_{n_{j-1}-1}^{j-1}$, $x_{0}^{j})\cup (x_0^j$, $P_j$, $x_i^j)\cup (x_i^j$, $x_{n_k-1}^k$, $x_{0}^1) \cup (x_{0}^1$, $P_1$, $x_{r}^1) \cup (x_{r}^1$, $x_{0}^2) $ is a cycle of length $l(\varphi(j,i))=\sum_{l=2}^{j-1} n_l+i+2+r+1$ for each $j\in [3$, $k-1]$ and each $i\in[0$, $n_{j}-1]$;
%\item $\psi(k,i)=P_2 \cup \cdots \cup P_{k-1}\cup (x_{n_{k-1}-1}^{k-1}$, $x_{n_k-(i+1)}^{k})\cup (x_{n_k-(i+1)}^k$, $C_k$, $x_{n_k-1}^k)\cup e_k \cup (x_{0}^1$, $P_1$, $x_{r}^1) \cup (x_{r}^1$, $x_{0}^2)$ is a cycle of length $l(\psi(k,i))=\sum_{l=2}^{k-1} n_l+i+1+r+1$ for each $i\in[0$, $n_k-1]$.
\item $\psi(k,i)=P_2 \cup \cdots \cup P_{k-1}\cup (x_{n_{k-1}-1}^{k-1}$, $x_{n_k-(i+1)}^{k})\cup (x_{n_k-(i+1)}^k$, $P_k$, $x_{n_k-1}^k)\cup (x_{n_{k}-1}^{k}$, $x_{0}^{1}) \cup (x_{0}^1$, $P_1$, $x_{r}^1) \cup (x_{r}^1$, $x_{0}^2)$ is a cycle of length $l(\psi(k,i))=\sum_{l=2}^{k-1} n_l+i+1+r+1$ for each $i\in[0$, $n_k-1]$.
\end{itemize}
Therefore, $D$ contains at least one cycle of length $l$ for each $l\in[r+3$, $\sum_{l=2}^{k} n_l+r+1]$.

To conclude the proof, it is sufficient to prove that $D$ contains a cycle of length $l$ for each $l\in [3$, $r+2]$ (recall that $r\in [0$, $n_1-1]$ and $n_1\leq n_i$ for each $i\in[2$, $k]$).

Now, consider the subdigraph of $D$ induced by $V(D_1)$ and $V(D_2)$, namely $D'=D\langle V(D_1) \cup V(D_2)\rangle$. As  $(D_1$, $D_2)\neq \emptyset$ and $(D_2$, $D_1)\neq \emptyset$ in $D$, $D'$ is strong and thus, by Theorem \ref{theorem classification}, $D'$ contains a cycle of length $l$ for each $l\in[3$, $n_1+1]$. Since $0\leq r \leq n_1-1$ and $D'$ is a subdigraph of $D$, we have that $D$ contains a cycle of length $l$ for each $l\in[3$, $r+2]$, as wanted.
\end{enumerate}
From both cases we have the result.
\end{enumerate}

\end{proof}

\section{Open problem}

In the previous section we proved that a strong digraph in the g.s. of $k$ vertex disjoint Hamiltonian diagraphs, $D_1$, $D_2$, \ldots, $D_k$, is vertex-pancyclic, pancyclic or Hamiltonian and contains cycles of several lengths. However, two questions remain to be answered: 

\begin{enumerate}
\item Is there a strongly connected digraph $D\in \oplus_{i=1}^k D_i$ which is pancyclic but not vertex-pancyclic?

\item Is there a strongly connected digraph $D\in \oplus_{i=1}^k D_i$ which is Hamiltonian and contains cycles of each length in  \begin{math}[3, \max\{\left(\sum_{i\in S} |V(D_i)|\right) + 1 \colon S\subset[1,~k] \text{ with } |S|=k-1\}]\end{math} but is not pancyclic?
\end{enumerate}

%\acknowledgements
%\label{sec:ack}
%We would like to thank the anonymous referees for their careful reading and useful comments that helped to improve, substantially, the present paper.

\nocite{*}
\bibliographystyle{abbrvnat}
% use the following instead if you encounter problems 
%\bibliographystyle{alpha}
%\bibliography{sample-dmtcs}
\bibliography{Pancyclism-g.s.}
\label{sec:biblio}

\end{document}